\documentclass[12pt,reqno]{amsart} 
\usepackage{amssymb}
\usepackage[T1]{fontenc}
\usepackage{mathrsfs}
\usepackage{enumerate}

\usepackage[usenames]{color}
\definecolor{darkgreen}{rgb}{0,0.7,0}
\definecolor{darkred}{rgb}{0.9,0,0}
\definecolor{darkblue}{rgb}{0,0,0.9}
\definecolor{newercolor}{rgb}{0.2,0,1}
\definecolor{orange}{rgb}{0.9,0.5,0}

\newcommand{\es}{$^\dagger$}

\let\too=\longrightarrow

\reversemarginpar

\newcommand{\cj}[1][G]{\mathrel{\sim_{#1}}}
\newcommand{\cjle}[1][G]{\mathrel{\le_{#1}}}
\newcommand{\cjge}[1][G]{\mathrel{\ge_{#1}}}

\mathcode`\:="003A
\mathchardef\cdot="0201
\renewcommand{\*}{\mathop{\textup{\large{$*$}}}\limits}

\newcommand{\qq}{\mathfrak{q}}

\renewenvironment{enumerate}[1][]
{\begin{enumerat}[#1]\setlength{\itemsep}{6pt}}{\end{enumerat}}

\renewenvironment{itemize}
{\begin{itemiz}\setlength{\itemsep}{6pt}\setlength{\itemindent}{-20pt}}
{\end{itemiz}}

\newenvironment{enuma}{\begin{enumerate}[{\rm(a) }]}{\end{enumerate}}

\numberwithin{table}{section}

\newcommand{\boldd}[1]{{\mathversion{bold}\textbf{#1}}}

\newcommand{\bmid}{\mathrel{\big|}}

\newlength{\short}
\newcommand{\dbl}[2]{\renewcommand{\arraystretch}{1.0}%
\setlength{\arraycolsep}{0mm}%
\begin{array}{c}\rule{0pt}{12pt}#1\\#2\end{array}}

\newcommand{\2}[1]{\textbf{2\uppercase{#1}}}

\newcommand{\4}[1]{\widebar{#1}}
\newcommand{\5}[1]{\widehat{#1}}

\let\8=\scriptstyle
\newcommand{\9}[1]{{}^{#1}\!}

\def\pair[#1,#2]{[\hskip-1.5pt[#1,#2]\hskip-1.5pt]}
\def\trp[#1,#2,#3]{[\hskip-1.5pt[#1,#2,#3]\hskip-1.5pt]}

\let\oldcirc=\circ
\renewcommand{\circ}{\mathchoice
    {\mathbin{\scriptstyle\oldcirc}}{\mathbin{\scriptstyle\oldcirc}}
    {\mathbin{\scriptscriptstyle\oldcirc}}
    {\mathbin{\scriptscriptstyle\oldcirc}}}

\def\beq#1\eeq{\begin{equation*}#1\end{equation*}}
\def\beqq#1\eeqq{\begin{equation}#1\end{equation}}

\numberwithin{equation}{section}

\newtheorem{Thm}{Theorem}[section]
\newtheorem{Prop}[Thm]{Proposition}

\newtheorem{Lem}[Thm]{Lemma}
\newtheorem{Defi}[Thm]{Definition}

\newtheorem{Th}{Theorem}

\newcounter{let} 
\loop\stepcounter{let}
\expandafter\edef\csname cal\alph{let}\endcsname%
{\noexpand\mathcal{\Alph{let}}}
\ifnum\thelet<26\repeat

\newcommand{\calll}{\mathscr{L}}

\newcommand{\EE}{\mathbf{E}}
\newcommand{\EEE}[2][2]{\EE_{#1}(#2)}
\newcommand{\II}{\mathscr{I}}
\newcommand{\IND}[2]{\II_{#1}(#2)}

\newcommand{\tdef}[2][]{\expandafter\newcommand\csname#2\endcsname%
{#1\textup{#2}}}
\tdef{Iso}   \tdef{Aut}    \tdef{Out}    \tdef{Inn}    \tdef{Hom}
\tdef{End}   \tdef{Inj}    \tdef{map}    \tdef{Ker}    \tdef{Ob}
\tdef{Mor}   \tdef{Res}    \tdef{Id}     \tdef{Fr}     \tdef{Spin} 
\tdef{rk}    \tdef{conj}   \tdef{incl}   \tdef{proj}   \tdef{diag} 
\tdef{trf}   \tdef{Sol}    \tdef{Ext}    
\tdef{Rep}   \tdef{pr}    \tdef{Inndiag} \tdef{Outdiag}  \tdef{expt}
\tdef{supp}  \tdef{Isom}   \tdef{ord}    \tdef{Coker}  
\tdef[_]{typ} \tdef[_]{cat} \tdef[^]{op} \tdef[^]{ab}

\newcommand{\fdef}[1]{\expandafter\newcommand\csname#1\endcsname%
{\mathfrak{#1}}}
\fdef{X}  \fdef{red}  \fdef{foc}  \fdef{hyp}  \fdef{Lie} \fdef{Y}

\newcommand{\bbdef}[1]{\expandafter\newcommand%
\csname#1\endcsname{\mathbb{#1}}}
\bbdef{C} \bbdef{F} \bbdef{R} \bbdef{Z} \bbdef{N} \bbdef{Q} \bbdef{K}

\newcommand{\itdef}[1]{\expandafter\newcommand\csname#1\endcsname%
{\textit{#1}}}
\itdef{PSL}  \itdef{PSU}  \itdef{SL}   \itdef{SU}  \itdef{GL} \itdef{GU}
\itdef{Sp}   \itdef{PSp}  \itdef{PSO}  \itdef{SO}  \itdef{SD} \itdef{PGU} 
\itdef{PGL}  \itdef{UT}   \itdef{Fi}   \itdef{Co}  \itdef{GO}
\itdef{He}   \itdef{Sz}   \itdef{McL}  \itdef{Ru}  \itdef{Suz} 
\itdef{HS}   \itdef{Ly}   

\newcommand{\ON}{\textit{O'N}}

\newcommand{\PSSL}{\textit{P}\varSigma\textit{L}}
\newcommand{\GGL}{\varGamma\textit{L}}

\newcommand{\SSL}{\varSigma\textit{L}}

\newcommand{\gee}{\varepsilon}
\newcommand{\sminus}{\smallsetminus}
\newcommand{\lie}[3]{\def\test{#2}\def\tst{G}\ifx\test\tst{{}^{#1}#2_{#3}}
\else{{}^{#1}\!#2_{#3}}\fi}
\renewcommand{\*}{\,\lower6pt\hbox{\Large{\textup{*}}}\,}
\newcommand{\syl}[2]{\textup{Syl}_{#1}(#2)}
\newcommand{\sylp}[1]{\syl{p}{#1}}

\renewcommand{\Im}{\textup{Im}}
\newcommand{\autf}{\Aut_{\calf}}

\newcommand{\outf}{\Out_{\calf}}

\newcommand{\defeq}{\overset{\textup{def}}{=}}

\newcommand{\mxfoura}[8]{\left(\begin{smallmatrix}#1&#2&#3&#4\\#5&#6&#7&#8}
\newcommand{\mxfourb}[8]{\\#1&#2&#3&#4\\#5&#6&#7&#8\end{smallmatrix}\right)}

\let\emptyset=\varnothing
\renewcommand{\:}{\colon}
\newcommand{\pcom}{{}^\wedge_p}

\newcommand{\nsg}{\trianglelefteq}

\newcommand{\til}[1]{\widetilde{#1}}

\newcommand{\gen}[1]{{\langle}#1{\rangle}}

\newcommand{\longleft}[1]{\;{\leftarrow%
\count255=0 \loop \mathrel{\mkern-6mu}%
    \relbar\advance\count255 by1\ifnum\count255<#1\repeat}\;}
\newcommand{\longright}[1]{\;{\count255=0 \loop \relbar\mathrel{\mkern-6mu}%
    \advance\count255 by1\ifnum\count255<#1\repeat\rightarrow}\;}
\newcommand{\Right}[2]{\overset{#2}{\longright#1}}
\newcommand{\RIGHT}[3]{\mathrel{\mathop{\kern0pt\longright#1}
        \limits^{#2}_{#3}}}

\newcommand{\LEFT}[3]{\mathrel{\mathop{\kern0pt\longleft#1}\limits^{#2}_{#3}}
}
\newcommand{\dRIGHT}[3]{\mathrel{%
   \mathop{\vcenter{\baselineskip=0pt\hbox{$\kern0pt\longright#1$}%
   \hbox{$\kern0pt\longright#1$}}}\limits^{#2}_{#3}}}
\newcommand{\LRIGHT}[3]{\mathrel{%
   \mathop{\vcenter{\baselineskip=0pt\hbox{$\kern0pt\longleft#1$}%
   \hbox{$\kern0pt\longright#1$}}}\limits^{#2}_{#3}}}
\newcommand{\RLEFT}[3]{\mathrel{%
   \mathop{\vcenter{\baselineskip=0pt\hbox{$\kern0pt\longright#1$}%
   \hbox{$\kern0pt\longleft#1$}}}\limits^{#2}_{#3}}}
\newcommand{\onto}[1]{\;{\count255=0 \loop \relbar\mathrel{\mkern-6mu}%
    \advance\count255 by1
    \ifnum\count255<#1 \repeat \twoheadrightarrow}\;}

\newcommand{\allbold}[1]{\textbf{\boldmath{#1}}}

\newenvironment{spor}[1]{\noindent\allbold{$G$ $\cong$ #1 : }}{\medskip}

\newcommand{\abelS}{---}
\newcommand{\trivS}{---}

\title{Automorphisms of fusion systems of sporadic simple groups}

\author{Bob Oliver}
\address{Université Paris 13, Sorbonne Paris Cité, LAGA, UMR 7539 du CNRS, 
99, Av. J.-B. Clément, 93430 Villetaneuse, France.}
\email{bobol@math.univ-paris13.fr}
\thanks{B. Oliver is partially supported by UMR 7539 of the CNRS}


\begin{document} 

\begin{abstract} 
We prove here that with a very small number of exceptions, when $G$ is a 
sporadic simple group and $p$ is a prime such that the Sylow $p$-subgroups 
of $G$ are nonabelian, then $\Out(G)$ is isomorphic to the outer automorphism 
groups of the fusion and linking systems of $G$. In particular, the 
$p$-fusion system of $G$ is tame in the sense of \cite{AOV1}, and is tamely 
realized by $G$ itself except when $G\cong M_{11}$ and $p=2$. From the 
point of view of homotopy theory, 
these results also imply that $\Out(G)\cong\Out(BG\pcom)$ in many (but not 
all) cases.
\end{abstract}

\maketitle

This paper is centered around the comparison of certain outer automorphism 
groups associated to a sporadic simple group: outer automorphisms of the 
group itself, those of its fusion at different primes, and those of its 
classifying space completed at different primes. In most, but not all cases 
(under conditions made precise in Theorem \ref{t:tame}), these automorphism 
groups are all isomorphic. This comparison is important when studying 
extensions of fusion systems, and through that plays a role in Aschbacher's 
program (see, e.g., \cite{A-surv}) for reproving certain parts of the 
classification theorem from the point of view of fusion systems.

When $G$ is a finite group, $p$ is a prime, and $S\in\sylp{G}$, the 
\emph{$p$-fusion system} of $G$ is the category $\calf_S(G)$ whose objects 
are the subgroups of $G$, and which has morphism sets
	\[ \Mor_{\calf_S(G)}(P,Q) =
	\bigl\{\varphi\in\Hom(P,Q) \,\big|\, \varphi=c_x, ~ 
	\textup{some $x\in G$ with $xPx^{-1}\le Q$} \bigr\}. \]
A $p$-subgroup $P\le G$ is called \emph{$p$-centric} in $G$ if 
$Z(P)\in\sylp{C_G(P)}$; equivalently, if $C_G(P)=Z(P)\times C'_G(P)$ for 
some (unique) subgroup $C'_G(P)$ of order prime to $p$. 
The \emph{centric linking system} of $G$ at $p$ is the category 
$\call_S^c(G)$ whose objects are the subgroups of $S$ which are $p$-centric 
in $G$, and where
	\[ \Mor_{\call_S^c(G)}(P,Q) = T_G(P,Q)/C'_G(P) 
	\quad\textup{where}\quad
	T_G(P,Q) = \bigl\{x\in G \,\big|\, xPx^{-1}\le Q \bigr\}. \]
Note that there is a natural functor $\pi\:\call_S^c(G)\too\calf_S(G)$ 
which is the inclusion on objects, and which sends the class of $x\in 
T_G(P,Q)$ to $c_x\in\Hom(P,Q)$. Outer automorphism groups of these systems 
were defined in \cite{BLO1} and later papers (see below). We say that 
$\calf=\calf_S(G)$ is \emph{tamely realized} by $G$ if the natural 
homomorphism $\kappa_G\:\Out(G)\too\Out(\call_S^c(G))$ is surjective and 
splits. The fusion system $\calf$ is \emph{tame} if it is tamely realized 
by some finite group.

In terms of homotopy theory, it was shown in \cite[Theorem B]{BLO1} that for 
a finite group $G$ and $S\in\sylp{G}$, there is a natural isomorphism 
$\Out(\call_S^c(G))\cong\Out(BG\pcom)$. Here, $BG\pcom$ is the 
$p$-completion, in the sense of Bousfield-Kan, of the classifying space of 
$G$, and $\Out(X)$ means the group of homotopy classes of self equivalences 
of the space $X$. Thus $\calf_S(G)$ is tamely realized by $G$ if 
the natural map from $\Out(G)$ to $\Out(BG\pcom)$ is split surjective.


When $p=2$, our main result is easily stated: if $G$ is a sporadic 
simple group, then the 2-fusion system of $G$ is simple except when $G\cong 
J_1$, and is tamely realized by $G$ except when $G\cong M_{11}$. The 
2-fusion system of $M_{11}$ is tamely realized by $\PSL_3(3)$.

For $p$ odd, information about fusion systems of the sporadic groups at odd 
primes is summarized in Table \ref{tbl:odd-p}. In that table, for a given 
group $G$ and prime $p$ and $S\in\sylp{G}$,
\begin{itemize} 
\item a dash ``---'' means that $S$ is abelian or trivial;
\item ``constr.'' means that $\calf_S(G)$ is constrained; and 
\item an almost simple group $L$ in brackets means that $\calf_S(G)$ is 
almost simple but not simple, and is shown in \cite[16.10]{A-gfit} to be 
isomorphic to the fusion system of $L$.
\end{itemize}
For all other pairs $(G,p)$, $\calf$ is simple by \cite[16.10]{A-gfit}, and 
we indicate what is known about the nature of 
$\kappa_G\:\Out(G)\too\Out(\call_S^c(G))$. In addition, 
\begin{itemize} 
\item a dagger (\es) marks the pairs $(G,p)$ for which $S$ is extraspecial 
of order $p^3$.
\end{itemize}

\begin{table}[ht] 
\begin{center}
\renewcommand{\arraystretch}{1.3}
\begin{tabular}{|c|c||c|c|c|c|}
\hline
\allbold{$G$} & \allbold{$|\Out(G)|$} & \allbold{$p=3$}
& \allbold{$p=5$} & \allbold{$p=7$} & \allbold{$p\ge11$} \\ 
\hline\hline
$M_{12}$ & 2 & $\kappa$ isom.\es & \abelS & \trivS & \abelS \\
\hline
$M_{24}$ &1&  [$M_{12}:2$]\es & \abelS & \abelS & \abelS \\
\hline
$J_2$ &2& constr.\es & \abelS & \abelS & \abelS \\
\hline
$J_3$ &2& constr. & \abelS & \trivS & \abelS \\
\hline
$J_4$ &1& [$\lie2F4(2)$]\es & \abelS & \abelS & 11: constr.\es \\
\hline
$\Co_3$ &1& $\kappa$ isom. & constr.\es & \abelS & \abelS \\
\hline
$\Co_2$ &1& $\kappa$ isom. & constr.\es & \abelS & \abelS \\
\hline
$\Co_1$ &1& $\kappa$ isom. & [$\SO_5(5)$] & \abelS & \abelS \\
\hline
HS &2& \abelS & constr.\es & \abelS & \abelS \\
\hline
\McL &2& $\kappa$ isom. & constr.\es & \abelS & \abelS \\
\hline
Suz &2& $\kappa$ isom. & \abelS & \abelS & \abelS \\
\hline
He &2& $\Out(\call)=1$\es & \abelS & $\kappa$ isom.\es & \abelS \\
\hline
Ly &1& $\kappa$ isom. & $\kappa$ isom. & \abelS & \abelS \\
\hline
Ru &1& [$\lie2F4(2)$]\es & [$L_3(5):2$]\es & \abelS & \abelS \\
\hline
\ON &2& \abelS & \abelS & $\kappa$ isom.\es  & \abelS \\
\hline
$\Fi_{22}$ &2& $\kappa$ isom. & \abelS & \abelS & \abelS \\
\hline
$\Fi_{23}$ &1& $\kappa$ isom. & \abelS & \abelS & \abelS \\
\hline
$\Fi'_{24}$ &2& $\kappa$ isom. & \abelS & $\kappa$ isom.\es & \abelS \\
\hline
$F_5$ &2& $\kappa$ isom. & $\kappa$ isom. & \abelS & \abelS \\
\hline
$F_3$ &1& $\kappa$ isom. & $\kappa$ isom.\es & \abelS & \abelS \\
\hline
$F_2$ &1& $\kappa$ isom. & $\kappa$ isom. & \abelS & \abelS \\
\hline
$F_1$ &1& $\kappa$ isom. & $\kappa$ isom. & 
$\kappa$ isom. & 13: $\kappa$ isom.\es \\
\hline
\end{tabular}
\end{center}
\caption{Summary of results for odd $p$}
\label{tbl:odd-p}
\end{table}

Here, a fusion system $\calf=\calf_S(G)$ is \emph{constrained} if it contains a 
normal $p$-subgroup $Q\nsg\calf$ such that $C_S(Q)\le Q$. The fusion system 
$\calf$ is \emph{simple} if it has no proper nontrivial normal fusion 
subsystems. It is \emph{almost simple} if it contains a proper normal 
subsystem $\calf_0\nsg\calf$ which is simple, and such that 
$C_\calf(\calf_0)=1$. We refer to \cite[Definitions I.4.1 \& 
I.6.1]{AKO} for the definitions of normal $p$-subgroups and normal fusion 
subsystems, and to \cite[\S\,6]{A-gfit} for the definition of the 
centralizer of a normal subsystem.

Thus when $G$ is a sporadic simple group and $p$ is an odd prime such that 
the $p$-fusion system $\calf$ of $G$ is simple, we show in all cases that 
$\calf$ is tamely realized by $G$, and in fact that 
$\Out(G)\cong\Out(\call_S^c(G))$ except when $G\cong\He$ and $p=3$ (Theorem 
\ref{t:tame}).


Before going further, we need to define more precisely the automorphism 
groups which we are working with. All of the definitions given here apply to 
abstract fusion and linking systems (see, e.g., \cite[\S\,III.4.3]{AKO}), 
but for simplicity, we always assume that $\calf=\calf_S(G)$ and 
$\call=\call_S^c(G)$ for some finite group $G$ with $S\in\sylp{G}$. 

Automorphisms of $\calf=\calf_S(G)$ are straightforward. An automorphism 
$\alpha\in\Aut(S)$ is \emph{fusion preserving} if it induces an 
automorphism of the category $\calf$ (i.e., a functor from $\calf$ to 
itself which is bijective on objects and on morphisms). Set 
	\begin{align*} 
	\Aut(\calf)=\Aut(S,\calf) &= \bigl\{\alpha\in\Aut(S)\,\big|\, 
	\textup{$\alpha$ is fusion preserving}\bigr\} \\
	\Out(\calf)=\Out(S,\calf) &= \Aut(S,\calf)/\autf(S).
	\end{align*}
Here, by definition, $\autf(S)=\Aut_G(S)$: the automorphisms induced by 
conjugation in $N_G(S)$. These groups have been denoted $\Aut(S,\calf)$ and 
$\Out(S,\calf)$ in earlier papers, to emphasize that they are groups of 
automorphisms of $S$, but it seems more appropriate here to regard them as 
automorphisms of the \emph{fusion system} $\calf$ (as opposed to the category 
$\calf$).

Now assume $\call=\call_S^c(G)$. For each $P\in\Ob(\call)$, 
set $\iota_P=[1]\in\Mor_\call(P,S)$ (the ``inclusion'' of $P$ 
in $S$ in the category $\call$), and set 
$[\![P]\!]=\bigl\{[g]\,\big|\,g\in P\bigr\}\le\Aut_\call(P)$. Define
	\begin{align*} 
	\Aut(\call)=\Aut\typ^I(\call) &= \bigl\{ \beta\in\Aut\cat(\call) \,\big|\, 
	\beta(\iota_P)=\iota_{\beta(P)},~
	\beta([\![P]\!])=[\![\beta(P)]\!],~ \forall\, 
	P\in\calf^c \bigr\} \\
	\Out(\call)=\Out\typ(\call) &= \Aut(\call)\big/\gen{c_x\,|\,
	x\in N_G(S)}. 
	\end{align*}
Here, $\Aut\cat(\call)$ is the group of automorphisms of $\call$ as a 
category, and $c_x\in\Aut(\call)$ for $x\in N_G(S)$ sends $P$ to $\9xP$ and $[g]$ to 
$[\9xg]$. There are natural homomorphisms
	\[ \Out(G) \underset{\cong\Out(BG\pcom)}{\Right4{\kappa_G} 
	\Out(\call) \Right4{\mu_G}} \Out(\calf)
	\qquad\textup{and}\qquad \4\kappa_G=\mu_G\circ\kappa_G. \]
Here, $\kappa_G$ is defined by sending the class of $\alpha\in 
\Aut(G)$, chosen so that $\alpha(S)=S$, to the class of 
$\5\alpha\in\Aut(\call)$, where $\5\alpha(P)=\alpha(P)$ and 
$\5\alpha([g])=[\alpha(g)]$. For $\beta\in\Aut(\call)$, $\mu_G$ sends the 
class of $\beta\in\Aut(\call)$ to the class of 
	\[ \5\beta = \Bigl( S \RIGHT5{g\mapsto[g]}{\cong} [\![S]\!] 
	\RIGHT5{\beta|_{[\![S]\!]}}{\cong} 
	[\![S]\!] \RIGHT5{[g]\mapsto g}{\cong} S \Bigr) \in \Aut(\calf) \le 
	\Aut(S). \]
Then $\4\kappa_G\:\Out(G)\too\Out(\calf)$ is 
induced by restriction to $S$. See \cite[\S\,III.4.3]{AKO} or 
\cite[\S\,1.3]{AOV1} for more details on these definitions. 


By recent work of Chermak, Oliver, and Glauberman and Lynd, the nature 
of $\mu_G$ is now fairly well known in all cases.

\begin{Prop}[{\cite[Theorem C]{O-Ch}, \cite[Theorem 1.1]{GLynd}}] 
\label{mu-onto}
For each prime $p$, and each finite group $G$ with $S\in\sylp{G}$, 
$\mu_G\:\Out(\call_S^c(G))\Right2{}\Out(\calf_S(G))$ is surjective, and is 
an isomorphism if $p$ is odd. 
\end{Prop}

In fact, \cite{O-Ch} and \cite{GLynd} show that the conclusion of 
Proposition \ref{mu-onto} holds for all (abstract) fusion 
systems and associated linking systems. 

When $G$ is a sporadic simple group and $p$ is odd, a more direct proof 
that $\mu_G$ is an isomorphism is given in \cite[Propositions 4.1 \& 
4.4]{limz-odd}.

The fusion system $\calf=\calf_S(G)$ is \emph{tamely realized} by $G$ if 
$\kappa_G$ is split surjective, and is \emph{tame} if it is tamely realized 
by some finite group $G^*$ with $S\in\sylp{G^*}$ and $\calf=\calf_S(G^*)$. 
We refer to \cite[Theorems A \& B]{AOV1} or \cite[\S\,III.6.1]{AKO} for the 
original motivation for this definition. In practice, it is in many cases 
easier to study the homomorphism $\4\kappa_G$, which is why we include 
information about $\mu_G$ here. The injectivity of $\4\kappa_G$, when $p=2$ 
and $G$ is a sporadic simple group, follows from a theorem of Richard Lyons 
\cite[Theorem 1.1]{Lyons} (see the proof of Proposition \ref{kappa-2-big}). 

Fusion systems of alternating groups were shown to be tame in 
\cite[Proposition 4.8]{AOV1}, while those of finite groups of Lie type 
(including the Tits group) were shown to be tame in \cite[Theorems C \& 
D]{BMO2}. So the following theorem completes the study of tameness for 
fusion systems of the known finite nonabelian simple groups.


\begin{Th} \label{t:tame}
Fix a sporadic simple group $G$, a prime $p$ which divides $|G|$, and 
$S\in\sylp{G}$. Set $\calf=\calf_S(G)$ and $\call=\call_S^c(G)$. Then 
$\calf$ is tame. Furthermore, $\kappa_G$ and $\mu_G$ are isomorphisms 
(hence $\calf$ is tamely realized by $G$) if $p=2$, or if $p$ is odd and 
$S$ is nonabelian, with the following two exceptions:
\begin{enuma} 
\item $G\cong M_{11}$ and $p=2$, in which case $\Out(G)=1$ and 
$|\Out(\calf)|=|\Out(\call)|=2$; and 

\item $G\cong\He$ and $p=3$, in which case $|\Out(G)|=2$ and 
$\Out(\calf)=\Out(\call)=1$. 

\end{enuma}
\end{Th}

\begin{proof} By Proposition \ref{mu-onto}, $\mu_G$ is surjective in all 
cases, and is an isomorphism if $p$ is odd. When $p=2$, $\mu_G$ is 
injective (hence an isomorphism) by Propositions \ref{kappa-2-small} (when 
$|S|\le2^9$) and \ref{mu-2-big} (when $|S|\ge2^{10}$). Thus in all cases, 
$\kappa_G$ is an isomorphism if and only if $\4\kappa_G=\mu_G\circ\kappa_G$ 
is an isomorphism.

When $p=2$, $\4\kappa_G$ is an isomorphism, with the one exception $G\cong 
M_{11}$, by Propositions \ref{kappa-2-small} (when $|S|\le2^9$) and 
\ref{kappa-2-big} (when $|S|\ge2^{10}$). When $p$ is odd, $S$ is 
nonabelian, and $\calf$ is not simple, then $\4\kappa_G$ is an isomorphism 
by Proposition \ref{kappa-nonsimple}. When $p$ is odd and $\calf$ is 
simple, $\4\kappa_G$ is an isomorphism except when $G\cong\He$ and 
$p=3$ by Proposition \ref{kappa-odd}. The two exceptional cases are handled 
in Propositions \ref{kappa-2-small} and \ref{kappa-odd}. 
\end{proof}

In the first half of the paper, we compare $\Out(G)$ with $\Out(\calf)$: 
first listing some general results in Section \ref{s:background}, and then 
applying them to determine the nature of $\4\kappa_G$ in Sections 
\ref{s:kappa-2} (for $p=2$) and \ref{s:kappa-odd} (for $p$ odd). We then 
compare $\Out(\calf)$ with $\Out(\call)$ (when $p=2$) in the last half of 
the paper: general techniques for determining $\Ker(\mu_G)$ are listed in 
Section \ref{s:Ker(mu)}, and these are applied in Section \ref{s:Ker(mu)=1} 
to finish the proof of the main theorem. 

The author plans, in a future paper with Jesper Grodal, to look more 
closely at the fundamental groups of geometric realizations of the 
categories $\call_S^c(G)$ when $G$ is a sporadic group. This should give 
alternative proofs for several of the cases covered by Theorem 
\ref{t:tame}. 

I would like to thank Michael Aschbacher for explaining to me the potential 
importance of these results. Kasper Andersen made some computer 
computations several years ago involving the Rudvalis sporadic group at 
$p=2$; while they're not used here, they probably gave me hints as to how 
to proceed in that case (one of the hardest). I also thank the referee for 
his many suggestions which helped simplify or clarify several arguments. I would 
especially like to thank Richard Lyons for the notes \cite{Lyons} he wrote 
about automorphisms of sporadic groups, without which I might not have 
known how to begin this project. 

\textbf{Notation: } We mostly use Atlas notation \cite[\S\,5.2]{atlas} for 
groups, extensions, extraspecial groups, etc., as well as for names ($\2a$, 
$\2b$, $\3a$, \dots) of conjugacy classes of elements. An elementary 
abelian $2$-group has type $\2a^n$ if it is $\2a$-pure of rank $n$ 
(similarly for an elementary abelian $3$-group of type $\3a^n$); it has 
type $\2a_i\0b_j\dots$ if it contains $i$ elements of class $\2a$, $j$ of 
class $\2b$, etc. Also, $A_n$ and $S_n$ denote the alternating and 
symmetric groups on $n$ letters, $E_{p^k}$ (for $p$ prime) is an elementary 
abelian $p$-group of order $p^k$, and $\UT_n(q)$ (for $n\ge2$ and $q$ a 
prime power) is the group of upper triangular matrices in $\GL_n(q)$ with 
$1$'s on the diagonal. As usual, $G^\#=G\sminus\{1\}$ is the set of 
nonidentity elements of a group $G$, and $Z_2(S)\le S$ (for a $p$-group 
$S$) is the subgroup such that $Z_2(S)/Z(S)=Z(S/Z(S))$. For groups $H\le G$ 
and elements $g,h\in G$, $\9gh=ghg^{-1}$ and $\9gH=gHg^{-1}$. For each pair 
of groups $H\le G$, 
	\[ \Aut_G(H)=\{(x\mapsto \9gx)\,|\,g\in N_G(H)\}\le\Aut(H)
	\quad\textup{and}\quad \Out_G(H)=\Aut_G(H)/\Inn(H). \]

We assume in all cases the known order of $\Out(G)$ for sporadic groups 
$G$, without giving references each time.

\bigskip

\section{Automorphism groups of fusion systems: generalities}
\label{s:background}

We give here some techniques which will be used to determine the nature of 
$\4\kappa_G$. We begin with the question of injectivity. Recall that 
$|\Out(G)|\le2$ for each sporadic simple group $G$.

\begin{Lem} \label{l:kappa-inj}
Fix a prime $p$. Let $G$ be a finite group, fix $S\in\sylp{G}$, and set 
$\calf=\calf_S(G)$. 
\begin{enuma} 


\item For each $\alpha\in\Aut(G)$, the class $[\alpha]\in\Out(G)$ lies in  
$\Ker(\4\kappa_G)$ if and only if there is $\alpha'\in[\alpha]$ such that 
$|S|\bmid|C_G(\alpha')|$.

\item Assume $|\Out(G)|=2$ and $p$ is odd. If there is no 
$\alpha\in\Aut(G)$ such that $|\alpha|=2$ and $|S|\bmid|C_G(\alpha)|$, then 
$\4\kappa_G$ is injective.

\item Assume $|\Out(G)|=2$. If $\Out_{\Aut(G)}(Q)>\Out_G(Q)$ for some 
$Q\nsg S$, then $\4\kappa_G$ is injective. 

\end{enuma}
\end{Lem}

\begin{proof} \textbf{(a) } We can assume $\alpha$ is chosen so that 
$\alpha(S)=S$. If $[\alpha]\in\Ker(\4\kappa_G)$, then 
$\alpha|_S\in\Aut_G(S)$: conjugation by some $g\in N_G(S)$. Set 
$\alpha'=\alpha\circ c_g^{-1}\in\Aut(G)$; then $[\alpha']=[\alpha]$ in 
$\Out(G)$, and $C_G(\alpha')\ge S$.

Conversely, assume $|S|\bmid|C_G(\alpha')|$. Then $C_G(\alpha')\ge\9gS$ 
for some $g\in G$. Set $\alpha''=c_g\circ\alpha'\circ c_g$ (composing 
from right to left), where $c_g\in\Inn(G)$ and $c_g(S)=\9gS$. Then 
$[\alpha'']=[\alpha']=[\alpha]$ in $\Out(G)$, $\alpha''|_S=\Id_S$, 
and hence $\4\kappa_G([\alpha])=\4\kappa_G([\alpha''])=1$. 

\smallskip

\noindent\textbf{(b) } If $\4\kappa_G$ is not injective, then by (a), there 
is $\alpha\in\Aut(G)\sminus\Inn(G)$ such that $|S|\bmid|C_G(\alpha)|$. 
Since $|\Out(G)|=2$, $|\alpha|=2m$ for some $m\ge1$. Thus $|\alpha^m|=2$, 
and $|S|\bmid|C_G(\alpha^m)|$.

\smallskip

\noindent\textbf{(c) } If $Q\nsg S$ and $\Out_{\Aut(G)}(Q)>\Out_G(Q)$, then 
there is $\beta\in\Aut(G)\sminus\Inn(G)$ such that $\beta(Q)=Q$ and 
$\beta|_Q\notin\Aut_G(Q)$. Since $S\in\sylp{N_G(Q)}$, we can arrange that 
$\beta(S)=S$ by replacing $\beta$ by $c_x\circ\beta$ for some appropriate 
element $x\in N_G(Q)$. We still have $\beta|_Q\notin\Aut_G(Q)$, so 
$\beta|_S\notin\Aut_G(S)$, and $\4\kappa_G([\beta])\ne1$. Thus $\4\kappa_G$ 
is nontrivial, and is injective if $|\Out(G)|=2$. 
\end{proof}


A finite group $H$ will be called \emph{strictly $p$-constrained} if 
$C_H(O_p(H))\le O_p(H)$; equivalently, if $F^*(H)=O_p(H)$. 

\begin{Lem} \label{l:HtoF}
Fix a prime $p$. Let $G$ be a finite group, fix $S\in\sylp{G}$, and set 
$\calf=\calf_S(G)$. Let $H<G$ be a subgroup which contains $S$. 
\begin{enuma} 

\item If $H$ is strictly $p$-constrained, then $\kappa_H$ and $\mu_H$ are 
isomorphisms.

\item Assume $H=N_G(Q)$, where either $Q$ is characteristic in $S$, or 
$|Q|=p$, $Q\le Z(S)$, and $\Aut(\calf)$ sends each $G$-conjugacy class of 
elements of order $p$ in $Z(S)$ to itself. Set $\calf_H=\calf_S(H)$ for 
short, and set $\Aut^0(\calf_H)=\Aut(\calf_H)\cap\Aut(\calf)$ and 
$\Out^0(\calf_H)=\Aut^0(\calf_H)/\Aut_H(S)$. Then the inclusion of 
$\Aut^0(\calf_H)$ in $\Aut(\calf)$ induces a surjection of 
$\Out^0(\calf_H)$ onto $\Out(\calf)$, and hence 
$|\Out(\calf)|\le|\Out^0(\calf_H)|\le|\Out(\calf_H)|$. 

If in addition, $H$ is strictly $p$-constrained or $\4\kappa_H$ is onto, 
and we set $\Out^0(H)=\4\kappa_H^{-1}(\Out^0(\calf_H))$, then 
$|\Out(\calf)|\le|\Out^0(H)|\le|\Out(H)|$. 

\end{enuma}
\end{Lem}

\begin{proof} \textbf{(a) } See, e.g., \cite[Proposition 1.6(a)]{BMO2}.

\smallskip

\noindent\textbf{(b) } We first claim that 
	\beqq \Aut(\calf) = \Aut_G(S)\cdot N_{\Aut(\calf)}(Q)
	\le \Aut_G(S)\cdot\Aut^0(\calf_H) \label{e:HtoF} \eeqq
as subgroups of $\Aut(S)$. If $Q$ is characteristic in $S$, then the 
equality is clear. If $|Q|=p$, $Q\le Z(S)$, and each $\alpha\in\Aut(\calf)$ 
sends $Q$ to a subgroup which is $G$-conjugate to $Q$, then the equality 
follows from the Frattini argument (and since each subgroup of $Z(S)$ which 
is $G$-conjugate to $Q$ is $N_G(S)$-conjugate to $Q$). If 
$\alpha\in\Aut(S)$ normalizes $Q$ and preserves fusion in $G$, then it 
preserves fusion in $H=N_G(Q)$. Thus 
$N_{\Aut(\calf)}(Q)\le\Aut^0(\calf_H)$, proving the second relation in 
\eqref{e:HtoF}. 

Now, $\Aut_H(S)\le\Aut_G(S)\cap\Aut^0(\calf_H)$. Together with 
\eqref{e:HtoF}, this implies that the natural homomorphism
	\[ \Out^0(\calf_H) = \Aut^0(\calf_H)\big/\Aut_H(S) \Right5{}
	\Aut(\calf)\big/\Aut_G(S) = \Out(\calf) \]
is well defined and surjective. The last statement now follows from (a).
\end{proof}

The next lemma will be useful when determining $\Out(H)$ for the subgroups 
$H$ which appear when applying Lemma \ref{l:HtoF}(b).

\begin{Lem} \label{Out(Q:H)}
Let $H$ be a finite group, and let $Q\nsg H$ be a characteristic subgroup 
such that $C_H(Q)\le Q$. Set $H^*=\Out_H(Q)\cong H/Q$. 
\begin{enuma} 

\item There is an exact sequence 
	\[ 1 \Right2{} H^1(H^*;Z(Q)) \Right4{} \Out(H) \Right4{R} 
	N_{\Out(Q)}(H^*)\big/H^*, \]
where $R$ sends the class of $\alpha\in\Aut(H)$ to the class of 
$\alpha|_Q$. 

\item Assume $R\le Z(Q)$ and $R\nsg H$. Let $\alpha\in\Aut(H)$ be such that 
$\alpha|_R=\Id_R$ and $[\alpha,H]\le R$. Then there is 
$\psi\in\Hom_{H}(Q/R,R)$ such that $\alpha(g)=g\psi(gR)$ for each $g\in 
Q$, and hence $\alpha|_Q=\Id_Q$ if $\Hom_{H}(Q/R,R)=1$. If 
$\alpha|_Q=\Id_Q$, $[\alpha,H]\le R$, and $H^1(H^*;R)=0$, then 
$\alpha\in\Aut_R(H)$.

\item Fix a prime $p$, assume $Q$ is an extraspecial or elementary abelian 
$p$-group, and set $\4Q=Q/\Fr(Q)$. Set $H_0^*=O^{p'}(H^*)$, and 
$X=N_{\Out(Q)}(H^*)/H^*$. 
\begin{enumerate}[\rm(c.i) ] 
\item If $\4Q$ is absolutely irreducible as an $\F_pH^*$-module, then there 
is $Y\nsg X$ such that $Y\cong(\Z/p)^\times\big/Z(H^*)$ and $X/Y$ is 
isomorphic to a subgroup of $\Out(H^*)$. 
\item If $\4Q$ is absolutely irreducible as an $\F_pH_0^*$-module, then 
there is $Y\nsg X$ such that $Y\cong(\Z/p)^\times\big/Z(H^*)$ and 
	\[ |X/Y| \le |\Out(H_0^*)|\big/|\Out_{H^*}(H_0^*)|. \]


\end{enumerate}
Here, $Z(H^*)$ acts on $\4Q$ via multiplication by scalars, and we regard 
it as a subgroup of $(\Z/p)^\times$ in that way.

\end{enuma}
\end{Lem}

\begin{proof} \textbf{(a) } The exact sequence is a special case of 
\cite[Lemma 1.2]{OV2}. 

\noindent\textbf{(b) } By assumption, there is a function $\psi\:Q/R\too R$ 
such that $\alpha(g)=g\psi(gR)$ for each $g\in Q$, and $\psi$ is a 
homomorphism since $R\le Z(Q)$. For each $h\in H$, $\alpha(h)=rh$ for some 
$r\in R$. So for $g\in Q$, since $[r,Q]=1$, we get 
$\psi(\9hgR)=(\9hg)^{-1}\alpha(\9hg)= (\9hg)^{-1}\9{rh}(\alpha(g)) 
=\9h(g^{-1}\alpha(g)) =\9h\psi(gR)$. Thus $\psi\in\Hom_{H}(Q/R,R)$. 

If $\alpha|_Q=\Id$ and $[\alpha,H]\le R$, then there is $\chi\:H^*\too R$ 
such that $\alpha(g)=\chi(gQ)g$ for each $g\in H$. Then 
$\chi(ghQ)=\chi(gQ)\cdot\9g\chi(hQ)$ for all $g,h\in H$, so $\chi$ is a 
1-cocycle. If $H^1(H^*;R)=0$, then there is $r\in R$ such that 
$\chi(gQ)=r(\9gr)^{-1}$ for each $g\in H$, and $\alpha$ is conjugation by 
$r$.

\noindent\textbf{(c) } If $\4Q$ is absolutely irreducible as an 
$\F_pH^*$-module, then $C_{\Out(Q)}(H^*)\cong(\Z/p)^\times$ consists of 
multiplication by scalars (see \cite[25.8]{A-FGT}), so its image $Y$ in 
$X=N_{\Out(Q)}(H^*)/H^*$ is isomorphic to $(\Z/p)^\times\big/Z(H^*)$. Also, 
$X/Y\cong\Out_{\Out(Q)}(H^*)$: a subgroup of $\Out(H^*)$. This proves 
(c.i).

If $\4Q$ is absolutely irreducible as an $\F_pH_0^*$-module, let $Y$ 
be the image of $C_{\Out(Q)}(H_0^*)$ in $X=N_{\Out(Q)}(H^*)/H^*$. 
Then $Y\cong(\Z/p)^\times\big/Z(H^*)$ (by \cite[25.8]{A-FGT} again), and 
	\begin{align*} 
	|X/Y| &= |N_{\Out(Q)}(H^*)\big/
	C_{\Out(Q)}(H_0^*)\cdot H^*| 
	\le |N_{\Out(Q)}(H_0^*)|\big/
	|C_{\Out(Q)}(H_0^*)\cdot H^*| \\
	&= |\Aut_{\Out(Q)}(H_0^*)|\big/|\Aut_{H^*}(H_0^*)| \\
	&\le |\Aut(H_0^*)|\big/|\Aut_{H^*}(H_0^*)| 
	= |\Out(H_0^*)|\big/|\Out_{H^*}(H_0^*)|. 
	\end{align*}
This proves (c.ii).
\end{proof}

The next lemma provides some simple tools for showing that certain 
representations are absolutely irreducible.

\begin{Lem} \label{V_abs_irr}
Fix a prime $p$, a finite group $G$, and an irreducible $\F_pG$-module $V$. 
\begin{enuma} 
\item The module $V$ is absolutely irreducible if and only if 
$\End_{\F_pG}(V)\cong\F_p$.

\item If $\dim_{\F_p}(C_V(H))=1$ for some $H\le G$, then $V$ is absolutely 
irreducible.

\item Assume $H\le G$ is a subgroup such that $V|_H$ splits as a direct sum 
of absolutely irreducible pairwise nonisomorphic $\F_pH$-submodules. Then 
$V$ is absolutely irreducible.

\end{enuma}
\end{Lem}

\begin{proof}
\textbf{(a) } See, e.g., \cite[25.8]{A-FGT}.

\smallskip

\noindent\textbf{(b) } Set $\End_{\F_pG}(V)=K$: a finite extension of 
$\F_p$. Then $V$ can be considered as a $KG$-module, so $[K:\F_p]$ divides 
$\dim_{\F_p}(C_V(H))$ for each $H\le G$. Since there is $H$ with 
$\dim_{\F_p}(C_V(H))=1$, this implies $K=\F_p$, and so $V$ is absolutely 
irreducible by (a).

\smallskip

\noindent\textbf{(c) } The hypothesis implies that the ring 
$\End_{\F_pH}(V)$ is isomorphic to a direct product of copies of $\F_p$, 
one for each irreducible summand of $V|_H$. Since $\End_{\F_pG}(V)$ is a 
subring of $\End_{\F_pH}(V)$, and is a field since $V$ is irreducible, it 
must be isomorphic to $\F_p$. So $V$ is absolutely irreducible by (a). 
\end{proof}

\begin{Lem} \label{l:H1}
Let $G$ be a finite group, and let $V$ be a finite $\F_pG$-module. 
\begin{enuma} 
\item If $C_V(O_{p'}(G))=0$, then $H^1(G;V)=0$.
\item If $|V|=p$, and $G_0=C_G(V)$, then 
$H^1(G;V)\cong\Hom_{G/G_0}\bigl(G_0/[G_0,G_0],V\bigr)$. 
\end{enuma}
\end{Lem}

\begin{proof} \textbf{(a) } Set $H=O_{p'}(G)$ for short. Assume $W\ge V$ is an 
$\F_pG$-module such that $[G,W]\le V$. Then $[H,W]=[H,V]=V$ since 
$C_V(H)=0$, and so $W=C_W(H)\oplus[H,W]=C_W(H)\oplus V$. 
Thus $H^1(G;V)\cong\Ext^1_{\F_pG}(\F_p,V)=0$.

Alternatively, with the help of the obvious spectral sequence, one can show 
that $H^i(G;V)=0$ for all $i\ge0$. 

\smallskip

\noindent\textbf{(b) } This is clear when $G$ acts trivially on $V$. It 
follows in the general case since for $G_0\nsg G$ of index prime to $p$ and 
any $\F_pG$-module $V$, $H^1(G;V)$ is the group of elements fixed by the 
action of $G/G_0$ on $H^1(G_0;V)$. 
\end{proof}

We end with a much more specialized lemma, which is needed when working 
with the Thompson group $F_3$. 

\begin{Lem} \label{l:A9-rk8}
Set $H=A_9$. Assume $V$ is an $8$-dimensional $\F_2H$-module such 
that for each $3$-cycle $g\in H$, $C_V(g)=0$. Then $V$ is absolutely 
irreducible, $\dim(C_V(T))=1$ for $T\in\syl2{H}$, and $N_{\Aut(V)}(H)/H=1$.
\end{Lem}

\begin{proof} Consider the following elements in $A_9$:
	\begin{align*} 
	a_1&=(1\,2\,3), &  a_2&=(4\,5\,6), &  a_3&=(7\,8\,9), \\
	b_1&=(1\,2)(4\,5), &  b_2&=(1\,2)(7\,8), &  
	b_3&=(1\,2)(4\,7)(5\,8)(6\,9). 
	\end{align*}
Set $A=\gen{a_1,a_2,a_3}\cong E_{27}$ and $B=\gen{b_1,b_2,b_3}\cong D_8$. 
Set $\4V=\4\F_2\otimes_{\F_2}V$. As an $\4\F_2A$-module, $\4V$ splits as a 
sum of 1-dimensional submodules, each of which has character 
$A\too\4\F_2^\times$ for which none of the $a_i$ is in the kernel. There 
are eight such characters, they are permuted transitively by $B$, and so 
each occurs with multiplicity $1$ in the decomposition of $\4V$. Thus $\4V$ 
is $AB$-irreducible, and hence $H$-irreducible (and $V$ is absolutely 
irreducible). Also, $\dim_{\4\F_2}(C_{\4V}(B))=1$, so 
$\dim(C_V(B))=1$, and $\dim(C_V(T))=1$ since $C_V(T)\ne0$.

In particular, $C_{\Aut(V)}(H)\cong\F_2^\times=1$, and hence 
$N_{\Aut(V)}(H)/H$ embeds into $\Out(H)$. So if $N_{\Aut(V)}(H)/H\ne1$, 
then the action of $H$ extends to one of $\5H\cong S_9$. In that case, if 
we set $x=(1\,2)\in\5H$, then $C_V(x)$ has rank $4$ since $x$ inverts $a_1$ 
and $C_V(a_1)=0$. But the group $C_{\5H}(x)/x\cong S_7$ acts faithfully on 
$C_V(x)$, and this is impossible since $\GL_4(2)\cong A_8$ contains no 
$S_7$-subgroup. (This argument is due to Richard Lyons \cite{Lyons}.) 
\end{proof}

\bigskip

\section{Automorphisms of $2$-fusion systems of sporadic groups}
\label{s:kappa-2}

The main result in this section is that when $G$ is a sporadic simple group 
and $p=2$, $\Out(\calf)\cong\Out(G)$ in all cases except when $G\cong 
M_{11}$. The first proposition consists mostly of the cases where this 
was shown in earlier papers.

\begin{Prop} \label{kappa-2-small}
Let $G$ be a sporadic simple group whose Sylow $2$-subgroups have order at 
most $2^9$. Then the $2$-fusion system of $G$ is tame. More precisely, 
$\kappa_G$ and $\mu_G$ are isomorphisms except when $G\cong M_{11}$, in 
which case the $2$-fusion system of $G$ is tamely realized by $\PSL_3(3)$.
\end{Prop}

\begin{proof} Fix $G$ as above, choose $S\in\syl2{G}$, and set 
$\calf=\calf_S(G)$. There are eleven cases to consider.

\boldd{If $G\cong M_{11}$,} then $\Out(G)=1$. Also, $\calf$ is the unique 
simple fusion system over $\SD_{16}$, so by \cite[Proposition 4.4]{AOV1}, 
$|\Out(\calf)|=2$, and $\kappa_{G^*}$ is an isomorphism for 
$G^*=\PSU_3(13)$ (and $\mu_{G^*}$ is an isomorphism by the proof of that 
proposition). Note that we could also take $G^*=\PSL_3(3)$. 

\boldd{If $G\cong J_1$,} then $\Out(G)=1$. Set $H=N_G(S)$. Since $S\cong 
E_8$ is abelian, fusion in $G$ is controlled by $H\cong2^3:7:3$, and so 
$\calf=\calf_S(H)$ and $\call\cong\call_S^c(H)$. Since $H$ is strictly 
2-constrained, $\Out(\call)\cong\Out(\calf)\cong\Out(H)=1$ by Lemma 
\ref{l:HtoF}(a), and so $\kappa_G$ and $\mu_G$ are isomorphisms. 

\boldd{If $G\cong M_{22}$, $M_{23}$, $J_2$, $J_3$, or $\McL$,} then $\calf$ 
is tame, and $\kappa_G$ is an isomorphism, by \cite[Proposition 4.5]{AOV1}. 
Also, $\mu_G$ was shown to be an injective in the proof of that 
proposition, and hence is an isomorphism by Proposition \ref{mu-onto}.

\boldd{If $G\cong M_{12}$, $\Ly$, $\HS$, or $\ON$,} then $\calf$ is tame, 
and $\kappa_G$ and $\mu_G$ are isomorphisms, by \cite[Lemmas 4.2 \& 5.2 and 
Proposition 6.3]{AOV3}. 
\end{proof}

It remains to consider the larger cases.

\begin{Prop} \label{kappa-2-big}
Let $G$ be a sporadic simple group whose Sylow $2$-subgroups have order at 
least $2^{10}$. Then $\4\kappa_G$ is an isomorphism.
\end{Prop}

\begin{proof} Fix $G$ as above, choose $S\in\syl2{G}$, and set 
$\calf=\calf_S(G)$. There are fifteen groups to consider, listed in Table 
\ref{tbl:p=2}.

We first check that $\4\kappa_G$ is injective in all cases. This follows 
from a theorem of Richard Lyons \cite[Theorem 1.1]{Lyons}, which says that 
if $\Out(G)\ne1$, then there is a $2$-subgroup of $G$ whose centralizer in 
$\Aut(G)=G.2$ is contained in $G$ \cite[Theorem 1.1]{Lyons}. Since that 
paper has not been published, we give a different argument here: one which is 
based on Lemma \ref{l:kappa-inj}(c), together with some well known (but 
hard-to-find-referenced) descriptions of certain subgroups of $G$ and of 
$\Aut(G)$.


The groups $G$ under consideration for which 
$|\Out(G)|=2$ are listed in Table \ref{tbl:p=2-aut}. In each case, 
$N_G(R)$ has odd index in $G$ (hence $R$ can be assumed to be normal in 
$S$), and $\Out_{\Aut(G)}(R)>\Out_G(R)$. So 
$\4\kappa_G$ is injective by Lemma \ref{l:kappa-inj}(c). 
\begin{table}[ht] 
\[ 
\renewcommand{\arraystretch}{1.4}
\begin{array}{|c||c|c|c|c|c|}
\hline
G & \Suz & \He & \Fi_{22} & \Fi_{24}' & F_5 \\\hline \hline
R & 2^{1+6}_- & 2^{4+4} & 2^{5+8} & 2^{1+12}_+ & 2^{1+8}_+ \\\hline
\Out_G(R) & \Omega_6^-(2) & S_3\times S_3 & S_3\times A_6 & 3\cdot U_4(3).2 & 
(A_5\times A_5).2 \\\hline
\Out_{G.2}(R) & \SO_6^-(2) & 3^2:D_8 & S_3\times S_6 & 3\cdot U_4(3).2^2 & 
(A_5\times A_5).2^2 \\\hline
\textup{Reference} & \textup{{\cite[p.56]{GL}}} 
& \textup{{\cite[\S\,5]{Wilson-aut}}}
& \textup{{\cite[37.8.2]{A-3tr}}} 
& \textup{{\cite[Th.E]{Wilson-Fi}}} 
& \textup{{\cite[Th.2]{NW-HN}}} \\\hline
\end{array} \]
\caption{} \label{tbl:p=2-aut}
\end{table}

It remains to prove that $|\Out(\calf)|\le|\Out(G)|$. Except 
when $G\cong\Ru$, we do this with the help of Lemma 
\ref{l:HtoF}(b) applied with $H$ as in Table \ref{tbl:p=2}. Set $Q=O_2(H)$, 
$\4Q=Q/Z(Q)$, and $H^*=\Out_H(Q)$. 


\begin{table}[ht] 
\[ 
\renewcommand{\arraystretch}{1.4}
\begin{array}{|c||c|c|c|c|c|l|}
\hline
G &|S| & H & |\Out(H)| & |\Out(\calf)| & |\Out(G)| & \textup{Reference}
\\
\hline
\hline
M_{24} &2^{10}& 2^{1+6}_+.L_3(2) & 2=1\cdot2 & 1^* & 1 & 
\textup{\cite[Lm. 39.1.1]{A-spor}} \\
\hline
J_4 &2^{21}& 2^{1+12}_+.3M_{22}:2 & 2=2\cdot1 & 1^* & 1 & \textup{\cite[\S\,1.2]{KW-J4}} \\ 
\hline
\Co_3 &2^{10}& 2\cdot\Sp_6(2) & 1 & 1 & 1 & \textup{\cite[Lm. 
4.4]{Finkelstein-Co3}} \\
\hline
\Co_2 &2^{18}& 2^{1+8}_+.\Sp_6(2) & 1=1\cdot1 & 1 & 1 & 
\textup{\cite[pp.113--14]{Wilson-Co2}} \\ 
\hline
\Co_1 &2^{21}& 2^{11}.M_{24} & 1=1\cdot1 & 1 & 1 & \textup{\cite[Lm. 
46.12]{A-spor}} \\
\hline
\Suz &2^{13}& 2^{1+6}_-.U_4(2) & 2=1\cdot2 & 2 & 2 & \textup{\cite[\S\,2.4]{Wilson-Suz}} \\ 
\hline
\He &2^{10}& 2^{1+6}_+:L_3(2) & 2=1\cdot2 & 2 & 2 & \textup{\cite[p. 253]{Held}} \\
\hline
\Ru &2^{14}& \dbl{2^{3+8}.L_3(2)}{2.2^{4+6}.S_5} &  & 
1^* & 1 
&\dbl{\textup{\cite[12.12]{A-over}}\hfill}{\textup{\cite[Th. J.1.1]{AS}}}\hfill \\ 
\hline
\Fi_{22} &2^{17}& 2^{10}.M_{22} & 2=1\cdot2 & 2 & 2 & 
\textup{\cite[25.7]{A-3tr}} \\
\hline
\Fi_{23} &2^{18}& 2^{11}.M_{23} & 1=1\cdot1 & 1 & 1 & 
\textup{\cite[25.7]{A-3tr}} \\
\hline
\Fi_{24}' &2^{21}& 2^{11}.M_{24} & 2=2\cdot1 & 2 & 2 & 
\textup{\cite[34.8, 34.9]{A-3tr}} \\
\hline
F_5 &2^{14}& 2^{1+8}_+.(A_5\times A_5).2 & 4=2\cdot2 & 2^* & 2 & 
\textup{\cite[\S\,3.1]{NW-HN}} \\ 
\hline
F_3 &2^{15}& 2^{1+8}_+.A_9 & 1=1\cdot1 & 1 & 1 & \textup{\cite[Thm. 2.2]{Wilson-Th}} \\ 
\hline
F_2 &2^{41}& 2^{1+22}_+.\Co_2 & 1=1\cdot1 & 1 & 1 & \textup{\cite[Thm. 2]{MS}} \\ 
\hline
F_1 &2^{46}& 2^{1+24}_+.\Co_1 & 1=1\cdot1 & 1 & 1 & \textup{\cite[Thm. 1]{MS}} \\ 
\hline
\end{array} \]
\caption{} \label{tbl:p=2}
\end{table}

\boldd{When $G\cong\Co_3$,} and $H=N_G(Z(S))\cong2\cdot\Sp_6(2)$ is 
quasisimple, $\Out(H)=1$ since $\Out(H/Z(H))=1$ by Steinberg's theorem (see 
\cite[Theorem 2.5.1]{GLS3}). Also, $\kappa_{H/Z(H)}$ is surjective by 
\cite[Theorem A]{BMO2}, so $\kappa_H$ and $\4\kappa_H$ are surjective by 
\cite[Proposition 2.18]{AOV1}. Hence $|\Out(\calf)|\le|\Out(H)|=1$ by Lemma 
\ref{l:HtoF}(b). 

\boldd{If $G\cong\Co_1$, $\Fi_{22}$, $\Fi_{23}$, or $\Fi_{24}'$,} then $Q$ is 
elementary abelian, $H^*\cong M_k$ for $k=24$, 22, 23, or 24, respectively, 
and $Q$ is an absolutely irreducible $\F_2H^*$-module by 
\cite[22.5]{A-3tr}. Also, $Q=J(S)$ (i.e., $Q$ is the unique abelian 
subgroup of its rank) in each case: by \cite[Lemma 46.12.1]{A-spor} when 
$G\cong\Co_1$, and by \cite[Exercise 11.1, 32.3, or 34.5]{A-3tr} when $G$ is 
one of the Fischer groups. By \cite[Lemma 4.1]{MSt} (or by 
\cite[22.7--8]{A-3tr} when $G$ is a Fischer group), $H^1(H^*;Q)$ has order 
$2$ when $G\cong\Fi_{24}'$ (and $Q$ is the Todd module for $H^*$), and has 
order $1$ when $G$ is one of the other Fischer groups ($Q$ is again the 
Todd module) or $\Co_1$ ($Q$ is the dual Todd module). So 
	\[ |\Out(\calf)| \le |\Out(H)| \le |H^1(H^*;Q)|\cdot|\Out(H^*)| 
	= |\Out(G)|\,: \]
the first inequality by Lemma
\ref{l:HtoF}(b), the second by Lemmas \ref{Out(Q:H)}(a) and \ref{Out(Q:H)}(c.i), 
and the equality by a case-by-case check (see Table \ref{tbl:p=2}). 

In each of the remaining cases covered by Table \ref{tbl:p=2}, 
$H=N_G(Z(S))$ and is strictly $2$-constrained, and $Q$ is extraspecial. We 
apply Lemma \ref{Out(Q:H)}(a) to get an upper bound for $|\Out(H)|$. This 
upper bound is listed in the fourth column of Table \ref{tbl:p=2} in the 
form $m=a\cdot b$, where $|H^1(H^*;Z(Q))|\le a$ and 
$|N_{\Out(Q)}(H^*)/H^*|\le b$. By Lemma \ref{l:H1}(b), 
$H^1(H^*;Z(Q))\cong\Hom(H^*,C_2)=1$ except when $G\cong J_4$ 
or $F_5$, in which cases it has order $2$. This explains the first factor 
in the fourth column. The second factor will be established case-by-case, 
as will be the difference between $|\Out(\calf)|$ and $|\Out(H)|$ when 
there is one (noted by an asterisk).

\boldd{If $G\cong M_{24}$ or $\He$,} then $H\cong2^{1+6}_+.L_3(2)$, and 
$\4Q$ splits as a sum of two nonisomorphic absolutely irreducible 
$\F_2H^*$-modules which differ by an outer automorphism of $H^*$. Hence 
$N_{\Out(Q)}(H^*)\cong L_3(2):2$, and 
$|\Out(H)|\le|N_{\Out(Q)}(H^*)/H^*|=2$. These two irreducible submodules in 
$\4Q$ lift to rank 4 subgroups of $Q$, of which exactly one is radical 
(with automizer $\SL_4(2)$) when $G\cong M_{24}$ (see \cite[Lemma 
40.5.2]{A-spor}). Since an outer automorphism of $H$ exchanges these two 
subgroups, it does not preserve fusion in $G$ when $G\cong M_{24}$, hence 
is not in $\Out^0(H)$ in the notation of Lemma \ref{l:HtoF}(b). So 
$|\Out(\calf)|\le|\Out^0(H)|=1$ in this case.

\boldd{If $G\cong J_4$,} then $H\cong2^{1+12}_+.3M_{22}:2$. The group 
$3M_{22}$ has a 6-dimensional absolutely irreducible representation over 
$\F_4$, which extends to an irreducible 12-dimensional representation of 
$3M_{22}:2$ realized over $\F_2$. (See \cite[p. 487]{KW-J4}: 
$3M_{22}<\SU_6(2)<\SO_{12}^+(2)$.) Hence $|\Out(H)|\le2$ by Lemmas 
\ref{Out(Q:H)}(a,c) and \ref{l:H1}(b), generated by the class of 
$\beta\in\Aut(H)$ of order $2$ which is the identity on $O^2(H)$ and on 
$H/Z(H)$. 

By \cite[Table 1]{KW-J4}, there is a four-group of type $\2{aab}$ in $H$, 
containing $Z(Q)=Z(H)$ (generated by an element of class $\2a$), whose 
image in $H/O_{2,3}(H)\cong M_{22}:2$ is generated by an outer involution 
of class $\2b$ in $\Aut(M_{22})$. Thus there are cosets of $Z(Q)$ in 
$H\sminus O^2(H)$ which contain $\2a$- and $\2b$-elements. Hence $\beta|_S$ 
is not $G$-fusion preserving, so $|\Out(\calf)|\le|\Out^0(H)|=1$ by Lemma 
\ref{l:HtoF}(b). 

\boldd{If $G\cong\Co_2$,} then $H=N_G(z)\cong2^{1+8}_+.\Sp_6(2)$. By 
\cite[Lemma 2.1]{FSmith}, the action of $H/Q$ on $\4Q$ is transitive on 
isotropic points and on nonisotropic points, and hence is irreducible. If 
$\4Q$ is not absolutely irreducible, then $\End_{\F_p[H/Q]}(\4Q)\ge\F_4$ by 
Lemma \ref{V_abs_irr}(a), so $H/Q\cong\Sp_6(2)$ embeds into $\SL_4(4)$, 
which is impossible since $\Sp_6(2)$ contains a subgroup of type $7:6$ 
while $\SL_4(4)$ does not.

Alternatively, $\4Q$ is absolutely irreducible by a theorem of Steinberg 
(see \cite[Theorem 2.8.2]{GLS3}), which says roughly that each irreducible 
$\4\F_2\Sp_6(\4\F_2)$-module which is ``small enough'' is still irreducible 
over the finite subgroup $\Sp_6(2)$.

Thus by Lemma \ref{Out(Q:H)}(c.i), $N_{\Out(Q)}(H^*)/H^*$ is 
isomorphic to a subgroup of $\Out(H^*)$, where 
$\Out(H^*)=1$ (see \cite[Theorem 2.5.1]{GLS3}). This 
confirms the remaining entries for $G$ in Table \ref{tbl:p=2}.

\boldd{If $G\cong\Suz$,} then $H\cong2^{1+6}_-.\Omega_6^-(2)$, $H^*$ has 
index $2$ in $\Out(Q)\cong\SO_6^-(2)$, and $|\Out(H)|\le2$ by Lemmas 
\ref{Out(Q:H)}(a) and \ref{l:H1}(b).

\boldd{If $G\cong F_5$,} then $H=N_G(z)\cong2^{1+8}_+.(A_5\wr2)$ for 
$z\in\2b$. As described in \cite[\S\,3.1]{NW-HN} and in \cite[Lemma 
2.8]{Harada-HN}, $O^2(H^*)$ acts on $Q$ as $\Omega_4^+(4)$ for some 
$\F_4$-structure on $\4Q$. Also, the $\2b$-elements in $Q\sminus Z(Q)$ are 
exactly those involutions which are isotropic under the $\F_4$-quadratic 
form on $\4Q\cong\F_4^4$.

Now, $H^*$ has index 2 in its normalizer $\SO_4^+(4).2^2$ in 
$\Out(Q)\cong\SO_8^+(2)$, so $|\Out(H)|\le4$ by Lemmas \ref{Out(Q:H)}(a) 
and \ref{l:H1}(b). Let $\beta\in\Aut(H)$ be the nonidentity automorphism 
which is the identity on $O^2(H)$ and on $H/Z(H)$. To see that 
$|\Out(\calf)|\le2$, we must show that $\beta$ does not preserve fusion in 
$S$. 

By \cite[p. 364]{NW-HN}, if $W=\gen{z,g}\cong E_4$ for $z\in Z(H)$ and 
$g\in H\sminus O^2(H)$, then $W$ contains an odd number of $\2a$-elements, 
and hence $g$ and $zg$ are in different classes (see also \cite[Lemma 
2.9.ii]{Harada-HN}). Hence $\beta$ is not fusion preserving since it 
doesn't preserve $G$-conjugacy classes. By Lemma \ref{l:HtoF}(b),
$|\Out(\calf)|\le|\Out^0(H)|\le2=|\Out(G)|$.

\boldd{If $G\cong F_3$,} then $H\cong2^{1+8}_+.A_9$. By 
\cite[\S\,3]{Parrott-Th}, the action of $A_9$ on $\4Q$ is not the 
permutation representation, but rather that representation twisted by the 
triality automorphism of $\SO_8^+(2)$. By \cite[3.7]{Parrott-Th}, if $x\in 
H^*\cong A_9$ is a $3$-cycle, then $C_{\4Q}(x)=1$. Hence we are in the 
situation of Lemma \ref{l:A9-rk8}, and $N_{\Out(Q)}(H^*)/H^*=1$ by that 
lemma. So $\Out(H)=1$ by Lemmas \ref{Out(Q:H)}(a) and \ref{l:H1}(b), and 
$\Out(\calf)=1$.

\boldd{If $G\cong F_2$ or $F_1$,} then $H=H_1\cong2^{1+22}_+.\Co_2$ or 
$2^{1+24}_+.\Co_1$, respectively. Set $Q=O_2(H)$ and $\4Q=Q/Z(Q)$. If 
$G\cong F_1$, then $\4Q\cong\til\Lambda$, the mod 2 Leech lattice, and is 
$\Co_1$-irreducible by \cite[23.3]{A-spor}. If $G\cong F_2$, then $\4Q\cong 
v_2^\perp/\gen{v_2}$ where $v_2\in\til\Lambda$ is the image of a 2-vector. 
The orbit lengths for the action of $\Co_2$ on $\til\Lambda/\gen{v_2}$ are 
listed in \cite[Table I]{Wilson-Co2}, and from this one sees that 
$v_2^\perp/\gen{v_2}$ is the only proper nontrivial $\Co_2$-linear subspace (the 
only union of orbits of order $2^k$ for $0<k<23$), and 
hence that $\4Q$ is $\Co_2$-irreducible. The absolute irreducibility of 
$\4Q$ (in both cases) now follows from Lemma \ref{V_abs_irr}(b), applied 
with $H=\Co_2$ or $U_6(2):2$, respectively.

Since $\Out(\Co_1)\cong\Out(\Co_2)=1$, $N_{\Out(Q)}(H^*)/H^*=1$ by Lemma 
\ref{Out(Q:H)}(c.i), and so $\Out(H)=1$ in both cases.

\smallskip

In in the remaining case, we need to work with two of the 2-local subgroups 
of $G$.

\boldd{Assume $G\cong\Ru$.} We refer to \cite[12.12]{A-over} and \cite[Theorem 
J.1.1]{AS} for the following properties. There are two conjugacy classes of 
involutions in $G$, of which the $\2a$-elements are 2-central. There 
are subgroups $H_1,H_3<G$ containing $S$ such that 
	\[ H_1\cong 2.2^{4+6}.S_5 \qquad 
	H_3\cong 2^{3+8}.L_3(2). \]
Set $Q_i=O_2(H_i)$ and $V_i=Z(Q_i)$; $V_1\cong C_2$ and $V_3\cong E_8$, 
and both are $\2a$-pure and normal in $S$. Also, $Q_1/V_1$ and $Q_3$ are special of types 
$2^{4+6}$ and $2^{3+8}$, respectively, and $Z(Q_3)$ and $Q_3/Z(Q_3)$ are 
the natural module and the Steinberg module, respectively, for 
$H_3/Q_3\cong\SL_3(2)$. 

Let $V_5<Q_1$ be such that $V_5/V_1=Z(Q_1/V_1)$. Then $V_5$ is of type 
$\2a^5$, and $C_{Q_1}(V_5)\cong Q_8\times E_{16}$. Also, $H_1/Q_1\cong S_5$, 
and $V_5/V_1$ and $Q_1/C_{Q_1}(V_5)$ are both natural modules for 
$O^2(H_1/Q_1)\cong\SL_2(4)$. Also, $V_3/V_1=C_{V_5/V_1}((S/Q_1)\cap 
O^2(H_1/Q_1))$: thus a 1-dimensional subspace of $V_5/V_1$ as an 
$\F_4$-vector space. 

The homomorphism $Q_1/C_{Q_1}(V_5)\too\Hom(V_5/V_1,V_1)$ which 
sends $g$ to $(x\mapsto[g,x])$ is injective and hence an isomorphism. So 
$Q_3\cap Q_1=C_{Q_1}(V_3)$ has index $4$ in $Q_1$, and hence 
$|Q_3Q_1/Q_3|=4$.

Fix $\beta\in\Aut(\calf)$. By Lemma \ref{l:HtoF}(a), for $i=1,3$, 
$\4\kappa_{H_i}$ is an isomorphism, so $\beta$ extends to an automorphism 
$\beta_i\in\Aut(H_i)$. Since $V_3$ is the natural module for 
$H_3/Q_3\cong\SL_3(2)$, $\beta_3|_{V_3}=c_x$ for some $x\in H_3$, and $x\in 
N_{H_3}(S)$ since $\beta_3(S)=S$. Then $x\in S$ since 
$N_{H_3/Q_3}(S/Q_3)=S/Q_3$, and upon replacing $\beta$ by 
$c_x^{-1}\circ\beta$ and $\beta_i$ by $c_x^{-1}\circ\beta_i$ ($i=1,3$), we 
can arrange that $\beta|_{V_3}=\Id$.

Since $\beta|_{V_3}=\Id$, $\beta_3$ also induces the identity on 
$H_3/Q_3\cong L_3(2)$ (since this acts faithfully on $V_3$), and on 
$Q_3/V_3\cong2^8$ (since this is the Steinberg module and hence absolutely 
irreducible). Since $Q_3/V_3$ is $H_3/Q_3$-projective (the Steinberg 
module), $H^1(H_3/Q_3;Q_3/V_3)=0$, so by Lemma \ref{Out(Q:H)}(b) (applied 
with $Q_3/V_3$ in the role of $R=Q$), the automorphism of $H_3/V_3$ induced 
by $\beta_3$ is conjugation by some $yV_3\in Q_3/V_3$. Upon replacing 
$\beta$ by $c_y^{-1}\circ\beta$ and similarly for the $\beta_i$, we can 
arrange that $[\beta_3,H_3]\le V_3$. 

Now, $Q_3/V_3\not\cong V_3$ are both irreducible $\F_2[H_3/Q_3]$-modules, 
so $\Hom_{H_3/Q_3}(Q_3/V_3,V_3)=0$. By Lemma \ref{Out(Q:H)}(b) again, 
applied this time with $Q_3\ge V_3$ in the role of $Q\ge R$, 
$\beta|_{Q_3}=\Id$. 

Now consider $\beta_1\in\Aut(H_1)$. Since $\beta_1$ is the identity on 
$Q_3=C_S(V_3)\ge C_{S}(V_5)=C_{H_1}(V_5)$, $\beta_1\equiv\Id_{H_1}$ 
modulo $Z(C_S(V_5))=V_5$ (since $c_g=c_{\beta_1(g)}\in\Aut(C_S(V_5))$ for 
each $g\in H_1$). So by Lemma \ref{Out(Q:H)}(b), there is 
$\psi\in\Hom_{H_1/Q_1}(Q_1/V_5,V_5/V_1)$ such that $\beta(g)\in 
g\psi(gV_5)$ for each $g\in Q_1$. Also, $\Im(\psi)\le V_3/V_1$ since 
$[\beta,S]\le V_3$, and hence $\psi=1$ since $V_5/V_1$ is irreducible. Thus 
$[\beta_1,Q_1]\le V_1$. 

We saw that $|Q_1Q_3/Q_3|=4$, so $\Aut_{Q_1}(V_3)$ is the group of all 
automorphisms which send $V_1$ to itself and induce the identity on 
$V_3/V_1$. Fix a pair of generators $uQ_3,vQ_3\in Q_1Q_3/Q_3$. Then 
$\beta(u)\in uV_1$ and $\beta(v)\in vV_1$, and each of the four possible 
automorphisms of $Q_3Q_1$ (i.e., those which induce the identity on $Q_3$ 
and on $Q_1Q_3/V_1$) is conjugation by some element of $V_3$ (unique 
modulo $V_1$). So after conjugation by an appropriate element of $V_3$, we 
can arrange that $\beta|_{Q_1Q_3}=\Id$ (and still $[\beta_3,H_3]\le V_3$).

Let $V_2<V_3$ be the unique subgroup of rank $2$ which is normal in $S$, 
and set $S_0=C_S(V_2)$. Thus $|S/S_0|=2$, and $S_0/Q_3\cong E_4$. Fix 
$w\in(S_0\cap Q_1Q_3)\sminus Q_3$ (thus $wQ_3$ generates the center of 
$S/Q_3\cong D_8$). Choose $g\in N_{H_3}(V_2)$ of order $3$; thus $g$ acts 
on $V_2$ with order $3$ and acts trivially on $V_3/V_2$. So 
$V_3\gen{g}\cong A_4\times C_2$, and since $|\beta_3(g)|=3$, we have 
$\beta_3(g)=rg$ for some $r\in V_2$. Set $w'=\9gw\in S_0$. Then 
$S_0=Q_3\gen{w,w'}$, $\beta(w)=w$ since $w\in Q_1Q_3$, and 
$\beta(w')=\beta(gwg^{-1})=rgwg^{-1}r^{-1}=\9rw'=w'$: the last equality 
since $w'\in S_0=C_S(V_2)$. Since $S=S_0Q_1$, this proves that 
$\beta=\Id_S$, and hence that $\Out(\calf)=1$.

This finishes the proof of Proposition \ref{kappa-2-big}.
\end{proof}


\bigskip

\section{Tameness at odd primes}
\label{s:kappa-odd}

We now turn to fusion systems of sporadic groups at odd primes, and first 
look at the groups whose $p$-fusion systems are not simple.

\begin{Prop} \label{kappa-nonsimple}
Let $p$ be an odd prime, and let $G$ be a sporadic simple group whose Sylow 
$p$-subgroups are nonabelian and whose $p$-fusion system is not simple. 
Then $\4\kappa_G$ is an isomorphism.
\end{Prop}

\begin{proof} Fix $S\in\sylp{G}$, and set $\calf=\calf_S(G)$. By 
\cite[16.10]{A-gfit}, if $\calf$ is not simple, then either $N_G(S)$ 
controls fusion in $G$ (``$G$ is $p$-Goldschmidt'' in the terminology of 
\cite{A-gfit}), in which case $S\nsg\calf$ and $\calf$ is constrained, or 
$\calf$ is almost simple and is realized by an almost simple group $L$ 
given explicitly in \cite[16.10]{A-gfit} and also in Table \ref{tbl:odd-p}. 
We handle these two cases separately.

\smallskip

\noindent\textbf{Case 1: } Assume first that $S\nsg\calf$ and hence $\calf$ 
is constrained. By \cite[Theorem 15.6]{A-gfit}, there are seven such cases 
$(G,p)$, also listed in Table \ref{tbl:odd-p}. By the tables in \cite[Table 
5.3]{GLS3}, in each case where $\Out(G)\ne1$, no involution of $\Aut(G)$ 
centralizes a Sylow $p$-subgroup. Thus $\4\kappa_G$ is injective in all 
seven cases by Lemma \ref{l:kappa-inj}(b). Set $H=N_G(H)/O_{p'}(N_G(H))$. 
Since $N_G(S)$ controls $p$-fusion in $G$,
	\beqq \Out(\calf)\cong\Out(H) \quad\textup{injects into}\quad 
	N_{\Out(S)}(\Out_G(S))/\Out_G(S) \,: \label{e:Out(F)u.b.} \eeqq
the isomorphism by Lemma \ref{l:HtoF}(a) and the injection by \ref{Out(Q:H)}(a). 

In the six cases described in Table \ref{tbl:constr.}, $S$ is extraspecial 
of order $p^3$ and exponent $p$.
\begin{table}[ht] 
\renewcommand{\arraystretch}{1.3}
\[ \begin{array}{|c||c|c|c|c|c|c|}
\hline
(G,p) & (J_2,3) & (\Co_3,5) & (\Co_2,5) & (\HS,5) & (\McL,5) & (J_4,11) 
\\\hline
|\Out(G)| & 2 & 1 & 1 & 2 & 2 & 1 \\\hline
\Out_G(S) & C_8 & C_{24}\rtimes C_2 & 4\cdot\Sigma_4 & C_8\rtimes C_2 & 
C_3\rtimes C_8 & 5\times2\cdot\Sigma_4 \\\hline
\end{array} \]
\caption{} \label{tbl:constr.}
\end{table}
Note that $\Out(S)\cong\GL_2(p)$. Using that $\PGL_2(3)\cong\Sigma_4$, 
$\PGL_2(5)\cong\Sigma_5$, and $\Sigma_4$ is maximal in $\PGL_2(11)$, we see 
that in all cases, $|\Out(\calf)|\le|\Out(G)|$ by \eqref{e:Out(F)u.b.}. So 
$\4\kappa_G$ and $\kappa_G$ are isomorphisms since they are injective. 

It remains to consider the case $(G,p)=(J_3,3)$, where $|S|=3^5$. Set 
$T=\Omega_1(S)$ and $Z=Z(S)$. By \cite[Lemma 5.4]{Janko}, $T\cong C_3^3$, 
$T>Z\cong C_3^2$, $Z\le[S,S]$, and there are two classes of elements 
of order $3$: those in $Z$ and those in $T\sminus Z$. Also, $S/Z$ 
is extraspecial of order $3^3$ with center $T/Z$, and $N_G(S)/S\cong 
C_8$ acts faithfully on $S/T$ and on $Z$. 

Consider the bilinear map
	\[ \Phi\: S/T \times T/Z \Right4{[-,-]} Z \]
where $\Phi(gT,hZ)=[g,h]$. This is nontrivial (otherwise we would have 
$T\le Z$), and hence is surjective since $N_G(S)/S\cong C_8$ acts 
faithfully on $Z$. Fix $x\in N_G(S)$ and $h\in T$ whose cosets generate 
the quotient groups $N_G(S)/S$ and $T/Z$, respectively. Since $x$ acts 
on $S/T\cong C_3^2$ with order $8$, it acts via an element of 
$\GL_2(3)\sminus\SL_2(3)$, and hence acts on $T/Z$ by inverting it (recall 
that $S/Z$ is extraspecial). So if we let $\Phi_h\:S/T\too Z$ be the 
isomorphism $\Phi_h(gT)=[g,h]$, then 
$\Phi_h(\9xgT)=[\9xg,h]=\9x[g,h^{-1}]=\9x\Phi_h(gT)^{-1}$. Thus if 
$\lambda,\lambda^3\in\F_9$ are the eigenvalues for the action of $x$ on 
$S/T$ (for some $\lambda$ of order $8$), then $\lambda^{-1},\lambda^{-3}$ 
are the eigenvalues for the action of $x$ on $Z$. So there is no nontrivial 
homomorphism $S/T\too Z$ that commutes with the actions of $x$.

Let $\alpha\in\Aut(\calf)$ be such that $\alpha|_{Z}=\Id$. Since 
$\alpha$ commutes with $\Phi$, it must either induce the identity on $S/T$ 
and on $T/Z$ or invert both quotient groups, and the latter is 
impossible since $S/Z$ is extraspecial. Since $\alpha$ is the identity 
on $Z$ and on $T/Z$, $\alpha|_T$ is conjugation by some element of 
$S$, and we can assume (modulo $\Inn(S)$) that $\alpha|_T=\Id$. Thus there 
is $\varphi\in\Hom(S/T,Z)$ such that $\alpha(g)=g\varphi(gT)$ for each 
$g\in S$, and $\varphi$ commutes with the action of $xS\in N_G(S)/S$. We 
just showed that this is only possible when $\varphi=1$, and conclude that 
$\alpha=\Id_S$.

Thus $\Aut(\calf)$ is isomorphic to a subgroup of 
$\Aut(Z)\cong\GL_2(3)$. Since $\Aut_G(S)\cong C_8$ acts faithfully on 
$Z$, and the Sylow 2-subgroups of $\GL_2(3)$ are semidihedral of order 
16, this shows that $|\Aut(\calf)|\le16$ and $|\Out(\calf)|\le2$. Since 
$\4\kappa_G$ is injective, it is an isomorphism.

\smallskip

\noindent\textbf{Case 2: } We now show that $\kappa_G$ is an isomorphism 
when $\calf$ is almost simple. Let $L$ be as in Table \ref{tbl:odd-p}. If 
$L\cong\lie2F4(2)$ and $p=3$, then $\Out(\calf)\cong\Out(L)=1=\Out(G)$ 
since $\4\kappa_L$ is an isomorphism by \cite[Proposition 6.9]{BMO2}. 

Otherwise, set $L_0=O^{p'}(L)$ and $\calf_0=\calf_S(L_0)$. By \cite[16.3 \& 
16.10]{A-gfit}, $\calf_0$ is simple, and hence $Z(\calf_0)=1$, when 
$L_0\cong M_{12}$ and $p=3$, and when $L_0\cong\Omega_5(5)$ or $\PSL_3(5)$ 
and $p=5$. Also, $\4\kappa_{L_0}$ is an isomorphism in these cases by 
Proposition \ref{kappa-odd} and \cite[Theorem A]{BMO2}, respectively, and 
$L\cong\Aut(L_0)$ and $|L/L_0|=2$ (hence $\Out(L)=1$) by 
\cite[16.10]{A-gfit}.
If $\Out(\calf)\ne1$, then there is $\alpha\in\Aut(\calf)\sminus\autf(S)$ 
such that $\alpha|_{S_0}=\Id$, and by the 
pullback square in \cite[Lemma 2.15]{AOV1}, this would lie in the image of a 
nontrivial element of $\Out(L)=1$. Thus $\Out(\calf)=1$, 
$\Out(G)=1$ by Table \ref{tbl:odd-p}, and so $\4\kappa_G$ and hence 
$\kappa_G$ are isomorphisms. 
\end{proof}

It remains to handle the cases $(G,p)$ where the $p$-fusion system of $G$ 
is simple.

\begin{Prop} \label{kappa-odd}
Let $p$ be an odd prime, and let $G$ be a sporadic simple group whose 
$p$-fusion system is simple. Then $\4\kappa_G$ is an isomorphism, except 
when $p=3$ and $G\cong\He$, in which case $|\Out(G)|=2$ and 
$|\Out(\calf_S(G))|=1$ for $S\in\syl3{G}$.
\end{Prop}

\begin{proof} Fix $G$ and $p$, choose $S\in\sylp{G}$, set 
$\calf=\calf_S(G)$, and assume $\calf$ is simple (see Table \ref{tbl:odd-p} 
or \cite[16.10]{A-gfit}). Set $\call=\call_S^c(G)$.

The centralizers of all involutions in $\Aut(G)$ are listed in, e.g., 
\cite[Tables 5.3a--z]{GLS3}. By inspection, for each pair 
$(G,p)$ in question other than $(\He,3)$ for which $\Out(G)\ne1$ (see 
Tables \ref{tbl:p^1+2} and \ref{tbl:p-odd}), there is no $\alpha\in\Aut(G)$ 
of order $2$ for which $|S|$ divides $|C_G(\alpha)|$. So by Lemma 
\ref{l:kappa-inj}(b), $\4\kappa_G$ is injective in all such cases. 


To prove that $\4\kappa_G$ is an isomorphism (with the one exception), it 
remains to show that $|\Out(\calf)|\le|\Out(G)|$.

\smallskip

\boldd{Assume $S$ is extraspecial of order $p^3$.} 
Set $H=N_G(S)$ and $H^*=\Out_G(S)\cong H/S$. 
We list in Table \ref{tbl:p^1+2} all pairs $(G,p)$ which occur, 
together with a description of $H^*$ and of $N_{\Out(S)}(H^*)$.
To determine $|N_{\Out(S)}(H^*)/H^*|$ in each case, just recall that 
$\GL_2(3)\cong2\cdot S_4\cong Q_8:S_3$, that $\PGL_2(5)\cong S_5$, and that 
when $p=7$ or $13$, each subgroup of order prime to $p$ in $\PGL_2(p)$ is 
contained in a subgroup isomorphic to $D_{2(p\pm1)}$ or $S_4$ (cf. 
\cite[Theorem 3.6.25]{Sz1}). 
\begin{table}[ht] 
\addtolength{\tabcolsep}{-1mm}
\renewcommand{\arraystretch}{1.3}
\[ \begin{array}{|c||c|c|c|c|c|c|c|}
\hline
(G,p) & (M_{12},3) & (\He,3) & (F_3,5) & (\He,7) & (\ON,7) & (\Fi_{24}',7) & 
(F_1,13) \\\hline
H^* & 2^2 & D_8 & 4\cdot S_4 & 3\times S_3 & 3\times D_8 & 
6\times S_3 & 3\times4\cdot S_4 \\\hline
N_{\Out(S)}(H^*) &D_8&\SD_{16}&4\cdot S_4&6\times S_3&3\times D_{16}
& C_6\wr C_2 & 3\times4\cdot S_4 \\\hline
|\Out(G)| &2&2&1&2&2&2&1 \\\hline
{\textup{Ref.}} & {\textup{\Small\cite[5.3b]{GLS3}}} &
{\textup{\Small\cite[3.9]{Held}}} & {\textup{\Small\cite[\S\,3]{Wilson-Th}}} &
{\textup{\Small\cite[3.23]{Held}}} & 
\multicolumn{2}{c|}{{\textup{\small\cite[Tbl.5.3s,v]{GLS3}}}} 
& {\textup{\Small\cite[\S\,11]{Wilson-M}}} \\\hline
\end{array} \]
\caption{} \label{tbl:p^1+2}
\end{table}

In all cases, we have 
	\[ |\Out(\calf)|\le |\Out(N_G(S))| \le |N_{\Out(S)}(H^*)/H^*|.  \]
The first inequality holds by Lemma \ref{l:HtoF}(b). The second holds by Lemma 
\ref{Out(Q:H)}(a), applied with $H=N_G(S)$, and since $H^1(H^*;Z(S))=0$ 
(Lemma \ref{l:H1}(b)). 
By Table \ref{tbl:p^1+2}, $|N_{\Out(S)}(H^*)/H^*|=|\Out(G)|$ in all cases. 
Hence $|\Out(\calf)|\le|\Out(G)|$, and so $\4\kappa_G$ is an isomorphism if 
it is injective. 

If $G\cong\He$ and $p=3$, then $H^*=\Out_G(S)\cong D_8$ permutes the four 
subgroups of index $3$ in $S\cong3^{1+2}_+$ in two orbits of two subgroups 
each. As described in \cite[Proposition 10]{Butler-He} (see also 
\cite[Table 5.3p, note 4]{GLS3}), the subgroups in one of the orbits are 
$\3a$-pure while those in the other have $\3a$- and $\3b$-elements, so no 
fusion preserving automorphism of $S$ exchanges them. So while 
$|N_{\Out(S)}(H^*)/H^*|=2$, we have $|\Out(\calf)|\le|\Out^0(H)|=1$ by 
Lemma \ref{l:HtoF}(b). Thus $\4\kappa_G$ is split surjective (and $G$ 
tamely realizes $\calf_S(G)$), but it is not an isomorphism.

\smallskip

\boldd{Assume $|S|\ge p^4$.}
Consider the subgroups $H<G$ described in Table \ref{tbl:p-odd}. In all 
cases, we can assume $H\ge S$.
\begin{table}[ht] 
\renewcommand{\arraystretch}{1.3}
\[ \begin{array}{|c|c|c||c|c|c|c|c|l|}
\hline
G & p & \textup{\Small{Case}} & H & \textup{\Small{$|\Out(H)|$}} & 
\textup{\Small{$|\Out(G)|$}} & K &N(-)& \textup{Reference} \\\hline
\Co_3 & 3 &4& 3^{1+4}_+.4S_6 & 1   & 1 && \3a & \textup{\cite[5.12]{Finkelstein-Co3}} \\\hline
\Co_2 & 3 &3b& 3^{1+4}_+.2^{1+4}_-.S_5 & 1   & 1 & 2^{1+4}_- &  \3a &
\textup{\cite[\S\,3]{Wilson-Co2}} \\\hline
\Co_1 & 3 &3a& 3^{1+4}_+.\Sp_4(3).2 & 1   & 1 & \Sp_4(3) &  \3c &
\textup{\cite[p.422]{Curtis-Co1}} \\\hline
\McL & 3 &2& 3^{1+4}_+.2S_5 & 2   & 2 & 2\cdot(5:4) &  \3a &
\textup{\cite[Lm.5.5]{Finkelstein-Co3}} \\\hline
\Suz & 3 &1& 3^5.M_{11} & 2 & 2 && J(S) &
\textup{\cite[Thm.]{Wilson-Suz}} \\\hline
\Ly & 3 &1& 3^5.(M_{11}\times2) & 1 & 1 && J(S) &
\textup{\cite[Tbl.I]{Lyons-Ly}} \\\hline
\Fi_{22} & 3 & 4 & 3^{1+6}_+.2^{3+4}.3^2.2 & 2 & 2 && \3b &
\textup{\cite[p.201]{Wilson-Fi22}} \\\hline
\Fi_{23} & 3 &4& 3^{1+8}_+.2^{1+6}_-.3^{1+2}_+.2S_4 & 1 & 1 &&  \3b &
\textup{\cite[\S\,1.2]{Wilson-Fi}} \\\hline
\Fi_{24}' & 3 &2& 3^{1+10}_+.U_5(2):2 & 2 & 2 & 2\cdot(11:10) &  \3b &
\textup{\cite[Th.B]{Wilson-Fi}} \\\hline
F_5 & 3 &4& 3^{1+4}_+.4A_5 & 2 & 2 &&  \3b &\textup{\cite[\S\,3.2]{NW-HN}} \\\hline
F_3 & 3 &5& \dbl{3^2.3^{3+4}.\GL_2(3)}{3.[3^8].\GL_2(3)} &  & 1 &  & 
\dbl{\3b^2}{\3b} &
\dbl{\textup{\cite[14.1--3]{A-over}}}{\textup{\cite[\S\S\,2,4]{Parrott-Th}}\hfill} \\\hline
F_2 & 3 &3b& 3^{1+8}_+.2^{1+6}_-.\SO_6^-(2) & 1  & 1 & 2^{1+6}_- &  \3b &
\textup{\cite[\S\,2]{Wilson-B}} \\\hline
F_1 & 3 &3a& 3^{1+12}_+.2\Suz.2 & 1  & 1 & 2\cdot(13:6) &  \3b &
\textup{\cite[\S\,3]{Wilson-M}} \\\hline
\Ly & 5 &4& 5^{1+4}_+.4S_6 & 1  & 1 & 2A_6 & \textbf{5A} &\textup{\cite[Tbl.I]{Lyons-Ly}} \\\hline
F_5 & 5 &3b& 5^{1+4}_+.2^{1+4}_-:5:4 & 2  & 2 & 2^{1+4}_- & \textbf{5B} &
\textup{\cite[\S\,3.3]{NW-HN}} \\\hline
F_2 & 5 &3b& 5^{1+4}_+.2^{1+4}_-.A_5.4 & 1 & 1 & 2^{1+4}_- & \textbf{5B} &
\textup{\cite[\S\,6]{Wilson-B}} \\\hline
F_1 & 5 &3a& 5^{1+6}_+.4J_2.2 & 1 & 1 & 2\cdot(7:6) & \textbf{5B} &
\textup{\cite[\S\,9]{Wilson-M}} \\\hline
F_1 & 7 &3a& 7^{1+4}_+.3\times2S_7 & 1 & 1 & 2\cdot(5:4) & \textbf{7B} &
\textup{\cite[\S\,10]{Wilson-M}} \\\hline
\end{array} \]
\caption{}
\label{tbl:p-odd}
\end{table}

\noindent\textbf{Case 1: } If $G\cong\Suz$ or $\Ly$ and $p=3$, then 
$H=N_G(J(S))$, where $J(S)\cong E_{3^5}$ and $H/J(S)\cong M_{11}$ or 
$M_{11}\times C_2$, respectively, and $|\Out(\calf)|\le|\Out(H)|$ by Lemma 
\ref{l:HtoF}(b). Set $V=O_3(H)=J(S)$ and $H^*=\Aut_H(V)\cong H/V$. Then 
$V$ is the Todd module for $O^2(H^*)\cong M_{11}$ (it contains $11$ 
subgroups of type $\3a$ permuted by $H^*$), so $H^1(H^*;V)=0$ by 
\cite[Lemma 4.1]{MSt}. Also, $V$ is absolutely $\F_3H^*$-irreducible since 
$H^*>11:5$. So by Lemma \ref{Out(Q:H)}(c.i) and since $\Out(M_{11})=1$, 
$|N_{\Aut(V)}(H^*)/H^*|\le2$ if $G\cong\Suz$ ($H^*\cong M_{11}$), and is 
trivial if $G\cong\Ly$. Lemma \ref{Out(Q:H)}(a) now implies that 
$|\Out(H)|\le2$ or $1$ for $G\cong\Suz$ or $\Ly$, respectively, and hence 
that $|\Out(\calf)|\le|\Out(G)|$.

For each of the remaining pairs $(G,p)$ displayed in Table \ref{tbl:p-odd}, 
except when $G\cong F_3$ and $p=3$ (Case 5), we set $Q=O_p(H)$, $\4Q=Q/Z(Q)$, 
$H^*=\Out_H(Q)$, $H_0=O^{p'}(H)$, and $H_0^*=\Out_{H_0}(Q)$.  Then $H$ is 
strictly $p$-constrained and $Q$ is extraspecial, and hence $Z(S)=Z(Q)$ has 
order $p$. Also, $H=N_G(Z(Q))=N_G(Z(S))$ by the above references, so 
$|\Out(\calf)|\le|\Out(H)|$ by Lemma \ref{l:HtoF}(b), and it remains to 
show that $|\Out(H)|\le|\Out(G)|$. By Lemma \ref{l:H1}(b), 
$H^1(H^*;Z(Q))=0$ in each of these cases, and hence $\Out(H)$ is sent 
injectively into the quotient group $N_{\Out(Q)}(H^*)/H^*$ by Lemma 
\ref{Out(Q:H)}(a). So it remains to show that 
$|N_{\Out(Q)}(H^*)/H^*|\le|\Out(G)|$.

\noindent\textbf{Case 2: } If $G\cong\McL$ or $\Fi_{24}'$ and $p=3$, then $\4Q$ 
is an absolutely irreducible $\F_pK$-module for $K\le H^*$ as given in 
Table \ref{tbl:p-odd}, and hence an absolutely irreducible 
$\F_pH^*$-module. So $|N_{\Out(Q)}(H^*)/H^*|\le2$ by Lemma 
\ref{Out(Q:H)}(c): since $|\Out(2S_5)|=2$ in the first case, and since 
$\Out(U_5(2).2)=1$ and $Z(U_5(2).2)=1$ in the second case. 

\noindent\textbf{Case 3: } If $G\cong F_1$ and $p=3$, then $\4Q$ splits as 
a sum of two absolutely irreducible $6$-dimensional $\F_3K$-modules. Since 
$5^2\bmid|\Suz|\bmid|H_0^*|$ while $5^2\nmid|\GL_6(3)|$, $\4Q$ is 
$H_0^*$-irreducible, hence absolutely $H_0^*$-irreducible by Lemma 
\ref{V_abs_irr}(c). In all other 
cases under consideration, $\4Q$ is easily checked to be 
an absolutely irreducible $\F_pK$-module for $K\le H_0^*$ as given in 
Table \ref{tbl:p-odd}, and hence an absolutely irreducible 
$\F_pH_0^*$-module. 

Thus $|\Out(\calf)|\le|N_{\Out(Q)}(H^*)/H^*|\le \eta\cdot 
|\Out(H_0^*)|\big/|\Out_{H^*}(H_0^*)|$ by Lemma \ref{Out(Q:H)}(c.ii), where 
for $Y$ as in the lemma, $\eta=|Y|=2$ when $(G,p)=(F_5,5)$ (and $H^*\ngeq 
Z(\Out(Q))$), and $\eta=|Y|=1$ otherwise.

\textbf{In Case 3a}, we have $\Out(H_0^*)=\Out_{H^*}(H_0^*)$ in all cases, 
so $|\Out(\calf)|=|\Out(H)|=1$. 

\textbf{In Case 3b}, we determine $\Out(H_0^*)$ by applying Lemma 
\ref{Out(Q:H)}(a) again, this time with $O_2(H_0^*)$ in the role of $Q$. 
Since $\Out(2^{1+4}_-)\cong S_5$ and $\Out(2^{1+6}_-)\cong\SO_6^-(2)$, the 
lemma implies that $\Out(H_0^*)=\Out_{H^*}(H_0^*)$ in each case, and hence 
that $|\Out(\calf)|\le\eta$.

\noindent\textbf{Case 4: } We show, one pair $(G,p)$ at a time, that 
$|N_{\Out(Q)}(H^*)/H^*|\le|\Out(G)|$ in each of these five cases.

\boldd{If $G\cong\Co_3$ and $p=3$,} then $Q\cong3^{1+4}_+$ and 
$\Out(Q)\cong\Sp_4(3):2$. Set $Z=Z(\Out(Q))\cong C_2$. Then 
$\Out(Q)/Z\cong\PSp_4(3):2\cong\SO_5(3)$ and $H^*/Z\cong C_2\times S_6$. 
Under this identification, the central involution $x\in Z(H^*/Z)$ acts as 
$-\Id_V\oplus\Id_W$ for some orthogonal decomposition $V\oplus W$ of the 
natural module $\F_3^5$; and since none of the groups $\Omega_2^{\pm}(3)$, 
$\Omega_3(3)$, or $\Omega_4^+(3)$ has order a multiple of $5$, $\dim(V)=4$ 
and $C_{\SO_5(3)}(x)\cong\GO_4^-(3)$. Since 
$\Omega_4^-(3)\cong\PSL_2(9)\cong A_6$, this shows that 
$C_{\Out(Q)/Z}(x)=H^*/Z\cong C_2\times S_6$. So $|N_{\Out(Q)}(H^*)/H^*|=1$.

\boldd{If $G\cong F_5$ and $p=3$,} then $Q\cong3^{1+4}_+$ and 
$H^*\cong4A_5$. By the argument in the last case, 
$N_{\Out(Q)}(Z(H^*))\cong4S_6$, so 
$|N_{\Out(Q)}(H^*)/H^*|=|N_{S_6}(A_5)/A_5|=2$.

\boldd{When $G\cong\Fi_{22}$ and $p=3$,} the subgroup 
$H\cong3^{1+6}_+.2^{3+4}.3^2.2$ is described in \cite[p. 201]{Wilson-Fi22}: 
$H^*$ can be regarded as a subgroup of $\GL_2(3)\wr S_3 < \Sp_6(3).2$. More 
precisely, $2^{3+4}<(Q_8)^3$ (recall $O_2(\GL_2(3))\cong Q_8$) is a 
subgroup of index $4$, one of the factors $C_3$ normalizes each $Q_8$ and 
the other permutes them cyclically, and the $C_2$ acts by inverting both 
factors $C_3$. Then $N_{\Out(Q)}(H^*)\le\GL_2(3)\wr S_3$ since it must 
permute the three $O_2(H^*)$-irreducible subspaces of $\4Q$, so 
$N_{\Out(Q)}(H^*)\cong2^{3+4}.(S_3\times S_3)$, and 
$|N(H^*)/H^*|=2$. 

\boldd{When $G\cong\Fi_{23}$ and $p=3$,} the subgroup $H$ is described in 
\cite[\S\,1.2]{Wilson-Fi}. The subgroup $R^*=O_2(H^*)\cong2^{1+6}_-$ has a 
unique faithful irreducible representation over $\F_3$, this is 
8-dimensional, and $N_{\SL_8(3)}(R^*)/R^*$ is sent injectively into 
$\Out(R^*)\cong\SO_6^-(2)\cong\SO_5(3)$. Since $H^*/R^*\cong3^{1+2}_+:2S_4$ 
is a maximal parabolic subgroup in $\SO_5(3)$, we get $N_{\Out(Q)}(H^*)/H^*=1$.


\boldd{If $G\cong\Ly$ and $p=5$,} then $\4Q$ is $\F_5[2A_6]$-irreducible 
since $3^2\nmid|\GL_3(5)|$, and is absolutely irreducible since $2A_6$ is 
not a subgroup of $\SL_2(25)$ (since $E_9$ is not a subgroup). Thus 
$|N_{\Out(Q)}(H^*)/H^*|\le|\Out(S_6)|=2$, with equality only if the action 
of $2A_6$ on $\4Q$ extends to $2A_6.2^2$. This is impossible, since the two 
classes of 3-elements in $2A_6$ act differently on $\4Q$ (note the 
action of a Sylow 3-subgroup on $\4Q$), so $N_{\Out(Q)}(H^*)/H^*=1$.


\noindent\textbf{Case 5: } \boldd{When $G\cong F_3$ and $p=3$,} we 
work with two different 3-local subgroups. Set $V_1=Z(S)$ and 
$V_2=Z_2(S)$, and set $H_i=N_G(V_i)$ and $Q_i=O_3(H_i)$ for $i=1,2$. By 
\cite[14.1.2 \& 14.1.5]{A-over} and \cite[\S\,4]{Parrott-Th}, $V_1\cong 
C_3$, $V_2\cong E_9$, $|Q_1|=|Q_2|=3^9$, and $H_1/Q_1\cong 
H_2/Q_2\cong\GL_2(3)$. Note that $S\le H_1\cap H_2$, and 
$|S|=3^{10}$. Also, the following hold: 
\begin{enumerate}[(1) ]

\item Set $V_5=Z_2(Q_2)$. Then $V_5=[Q_2,Q_2]\cong E_{3^5}$, $Q_2/V_5\cong 
E_{3^4}$, $V_2$ is the natural module for $G_2/Q_2\cong\GL_2(3)$, and 
$V_5/V_2$ is the projective absolutely irreducible $\PSL_2(3)$-module of 
rank $3$. Also, $V_5/V_2=Z(Q_2/V_2)$, and hence $Q_2/V_2$ is special of 
type $3^{3+4}$. See \cite[14.2]{A-over}.

\item By \cite[14.2.3]{A-over}, the quotient $Q_2/V_5$ is 
$G_2/Q_2$-indecomposable, and is an extension of one copy of the natural 
$\SL_2(3)$-module by another. Let $R_7<Q_2$ be such that $R_7>V_5$, and 
$R_7/V_5<Q_2/V_5$ is the unique $H_2/Q_2$-submodule of rank $2$ 
(thus $|R_7|=3^7$).

\item We claim that $C_{Q_2}(V_5)=V_5$. Assume otherwise: then 
$C_{Q_2}(V_5)\ge R_7$ since it is normal in $H_2$. So $V_5\le Z(R_7)$, and 
$|[R_7,R_7]|\le3$ since $R_7/V_5\cong E_9$. But $[R_7,R_7]<V_5$ is normal 
in $H_2$, so it must be trivial, and $R_7$ is abelian. This is impossible: 
$V_5$ contains elements of all three classes of elements of order $3$ 
\cite[14.2.2]{A-over}, while the centralizer of a $\3a$-element is isomorphic 
to $(3\times G_2(3)).2$ whose Sylow 3-subgroups are nonabelian of order 
$3^7$.

\item Set $V_3=Z_2(Q_1)$; then $V_3\cong E_{27}$, and $V_3/V_1$ is the 
natural module for $G_1/Q_1$ \cite[14.3.1]{A-over}. Since $V_3\nsg S$ and 
$V_2=Z_2(S)\cong E_9$, $V_3>V_2$. Also, $V_3/V_2\le Z(Q_2/V_2)=V_5/V_2$ 
since $|V_3/V_2|=3$. Thus $V_2<V_3<V_5$. 

By \cite[14.3.2]{A-over}, $[Q_1,Q_1]>V_3$, and $Q_1/[Q_1,Q_1]\cong E_{3^4}$ 
is $G_1/Q_1$-indecomposable and an extension of one copy of the natural 
$\SL_2(3)$-module by another. 

\item Set $W_7=C_G(V_3)\ge V_5$: a subgroup of $S$, hence of $Q_1\cap Q_2$, 
of order $3^7$ \cite[14.3.4]{A-over}. We claim that 
$W_7/V_5=Z(S/V_5)=C_{Q_2/V_5}(S/Q_2)$, 
where $S/V_5\cong C_3\times(C_3\wr C_3)$ by \cite[14.2.5]{A-over}. To 
see this, note that for each $g\in Q_2$ such that $gV_5\in 
C_{Q_2/V_5}(S/Q_2)$, the 
map $x\mapsto[x,g]$ is $S/Q_2$-linear from $V_5/V_2$ to $V_2$, so 
$V_3/V_2=[S,[S,V_5/V_2]]$ (see (1))
lies in its kernel. Thus $Z(S/V_5)\le W_7/V_5$, and they are equal since 
they both have order $9$.

\item To summarize, we have defined two sequences of subgroups
	\[ \hphantom{xxxxx} V_2 < V_5 < R_7 < Q_2 < H_2 
	\qquad\textup{and}\qquad V_1 < V_3 < W_7 < Q_1 < H_1, \]
those in the first sequence normal in $H_2$ and those in the second normal in 
$H_1$, where $V_m\cong E_{3^m}$ and $|R_7|=|W_7|=3^7$. 
In addition, $V_1<V_2<V_3<V_5<W_7<Q_2$. 

\end{enumerate}

Fix $\beta\in\Aut(\calf)$. By Lemma \ref{l:HtoF}(a), $\kappa_{H_2}$ is an 
isomorphism, and hence $\beta$ extends to an automorphism 
$\beta_2\in\Aut(H_2)$. Since $V_2$ is the natural module for 
$H_2/Q_2\cong\GL_2(3)$, $\beta_2|_{V_2}=c_x$ for some $x\in H_2$, and $x\in 
N_{H_2}(S)$ since $\beta_2(S)=S$. Upon replacing $\beta$ by 
$c_x^{-1}\circ\beta$ and $\beta_2$ by $c_x^{-1}\circ\beta_2$, we can 
arrange that $\beta|_{V_2}=\Id$. 

Since $\beta|_{V_2}=\Id$, $\beta_2$ also induces the identity on 
$H_2/Q_2\cong\GL_2(3)$ (since this acts faithfully on $V_2$), and induces 
$\gee\cdot\Id$ on $V_5/V_2\cong E_{27}$ for $\gee\in\{\pm1\}$ since it is 
absolutely irreducible. By (3), the homomorphism 
$Q_2/V_5\too\Hom(V_5/V_2,V_2)$ which sends $g$ to $(x\mapsto[g,x])$ is 
injective. Since $\beta$ induces the identity on $V_2$ and $\gee\cdot\Id$ 
on $V_5/V_2$, it also induces $\gee\cdot\Id$ on $Q_2/V_5$. By (1), 
$[Q_2/V_2,Q_2/V_2]=V_5/V_2$, so $\beta$ acts via the identity on $V_5/V_2$. 
Thus $\gee=+1$, and $\beta$ also induces the identity on $Q_2/V_5$. 

Now, $H^1(H_2/Q_2;Q_2/V_5)=0$ by Lemma \ref{l:H1}(a) (and since the central 
involution in $H_2/Q_2\cong\GL_2(3)$ inverts $Q_2/V_5$). So by Lemma 
\ref{Out(Q:H)}(b), applied with $H_2/V_5$ and $Q_2/V_5$ in the role of $H$ 
and $Q=R$, $\beta_2\equiv c_y$ modulo $V_5$ for some $y\in Q_2$. Upon 
replacing $\beta_2$ by $c_y^{-1}\circ\beta_2$, we can arrange that 
$[\beta,H_2]\le V_5$.

Next, note that $V_5/V_2=Z(Q_2/V_2)$ and 
$\Hom_{H_2/Q_2}(Q_2/V_5,V_5/V_2)=1$ by (1) and (2), and 
$H^1(H_2/Q_2;V_5/V_2)=0$ since $V_5/V_2$ is $H_2/Q_2$-projective. So by 
Lemma \ref{Out(Q:H)}(b), $\beta\equiv c_z$ (mod $V_2$) for some $z\in V_5$. 
Upon replacing $\beta_2$ by $c_z^{-1}\circ\beta_2$, we can now arrange that 
$[\beta_2,H_2]\le V_2$. 

By Lemma \ref{Out(Q:H)}(b), $\beta|_{Q_2}$ has the form 
$\beta(u)=u\chi(uV_2)$ for some $\chi\in\Hom_{H_2/Q_2}(Q_2/V_2,V_2)$. Also, 
$\chi$ factors through $Q_2/V_5$ since $[Q_2,Q_2]=V_5$ by (1). By (2), 
either $\chi=1$, or $\chi$ is surjective with kernel $R_7/V_2$. In either 
case, $\beta|_{R_7}=\Id$. Also, since $W_7/V_5=C_{Q_2/V_5}(S/Q_2)$ by (5), 
$\chi(W_7/V_5)\le C_{V_2}(S/Q_2)=V_1$. So $[\beta,W_7]\le V_1$.

By Lemma \ref{l:HtoF}(a) again, $\4\kappa_{H_1}$ is an isomorphism, and 
hence $\beta$ extends to $\beta_1\in\Aut(H_1)$. Let $\4\beta\in\Aut(S/V_1)$ 
and $\4\beta_1\in\Aut(H_1/V_1)$ be the automorphisms induced by $\beta$ and 
$\beta_1$. We have just shown that $\4\beta|_{W_7}=\Id$, and that 
$[\4\beta_1,S/V_1]\le V_2/V_1$. By Lemma \ref{Out(Q:H)}(b) again, 
$\4\beta|_{Q_1/V_1}$ has the form $\4\beta(g)=g\psi(gW_7)$ for some 
$\psi\in\Hom_{H_1/Q_1}(Q_1/W_7,V_3/V_1)$ with $\Im(\psi)\le V_2/V_1$. Since 
$Q_1/W_7$ and $V_3/V_1$ are natural modules for $\SL_2(3)$ by (5) and (4), 
$\psi$ must be surjective or trivial. Since $\psi$ is not 
surjective, $\4\beta|_{Q_1}=\Id$. Also, $H^1(H_1/Q_1;V_3/V_1)=0$ by 
Lemma \ref{l:H1}(a), so $\4\beta_1\in\Aut_{V_3/V_1}(H_1/V_1)$ by Lemma 
\ref{Out(Q:H)}(b). 

We can thus arrange, upon replacing $\beta_1$ by $c_w^{-1}\circ\beta_1$ for 
some $w\in V_3$, that $\4\beta_1=\Id$, and hence that $[\beta_1,H_1]\le 
V_1$. (We can no longer claim that $[\beta_2,H_2]\le V_2$, but this will 
not be needed.) Set $H'_1=[H_1,H_1]$. By (4), $H'_1\ge Q_1$ and 
$H'_1/Q_1\cong\SL_2(3)$. Also, $V_1=Z(H'_1)$, so $\beta_1|_{H_1'}$ has the 
form $\beta_1(g)=g\phi(g)$ for some $\phi\in\Hom(H'_1,V_1)$. But $H'_1$ is 
perfect by (4) again, so $\phi=1$, and $\beta_1=\Id$. Thus 
$\Out(\calf)=1$, and $\4\kappa_G$ is an isomorphism.

This finishes the proof of Proposition \ref{kappa-odd}.
\end{proof}

\bigskip

\section{Tools for comparing automorphisms of fusion and linking systems}
\label{s:Ker(mu)}

Throughout this section and the next, we assume $p=2$. Many of the 
definitions and statements given here are well known to hold for arbitrary 
primes, but we restrict to this case for simplicity. In particular, a 
strongly embedded subgroup $H<G$ always means a strongly 2-embedded 
subgroup; i.e., one such that $2\big||H|$ while $2\nmid|H\cap\9gH|$ for 
$g\in G\sminus H$.

\begin{Defi} \label{d:Zhat}
Fix a finite group $G$, choose $S\in\syl2{G}$, and set $\calf=\calf_S(G)$.
\begin{enuma} 
\item A subgroup $P\le S$ is \emph{fully normalized in $\calf$} if 
$N_S(P)\in\syl2{N_G(P)}$.

\item A $2$-subgroup $P\le G$ is \emph{essential} if $P$ is $2$-centric in 
$G$ (i.e., $Z(P)\in\syl2{C_G(P)}$), and $\Out_G(P)$ has a strongly 
embedded subgroup. Let $\EEE{G}$ be the set of all essential 
$2$-subgroups of $G$.

\item A subgroup $P\le S$ is \emph{$\calf$-essential} if $P$ is fully 
normalized in $\calf$ and essential in $G$. Let $\EE_\calf$ be the set of 
all $\calf$-essential subgroups of $G$.

\item $\5\calz(\calf) = \bigl\{W\le S\,\big|\, \textup{$W$ 
	elementary abelian, fully normalized in $\calf$,} \\
	\textup{\qquad\qquad\qquad$W=\Omega_1(Z(C_S(W)))$, $\autf(W)$ has 
	a strongly embedded subgroup}\bigr\}$.

\end{enuma}
\end{Defi}

Clearly, in the situation of Definition \ref{d:Zhat}, 
$\EE_\calf\subseteq\EEE{G}$, while each member of $\EEE{G}$ is 
$G$-conjugate to a member of $\EE_\calf$. If $W\in\5\calz(\calf)$ and 
$P=C_S(W)$, then by the following lemma, restriction defines a surjection 
from $\Out_G(P)$ onto $\Aut_G(W)$ with kernel of odd order. Hence 
$\Out_G(P)$ also has a strongly embedded subgroup, and $P\in\EE_\calf$. 

\begin{Lem} \label{Out->Aut}
Fix a finite group $G$ and $S\in\syl2{G}$, and set $\calf=\calf_S(G)$.
\begin{enuma} 

\item If $W\le P\le G$ are $2$-subgroups such that $W=\Omega_1(Z(P))$ 
and $P\in\syl2{C_G(W)}$, then restriction induces a surjection 
$\Out_G(P)\too\Aut_G(W)$ with kernel of odd order.

\item If $W\in\5\calz(\calf)$ and $P=C_S(W)$, then $P\in\EE_\calf$.
\end{enuma}
\end{Lem}

\begin{proof} \textbf{(a) } By the Frattini argument, $N_G(W)\le 
N_G(P)C_G(W)$, with equality since $W$ is characteristic in $P$. So the 
natural homomorphism
	\[ \Out_G(P) \cong N_G(P)/C_G(P)P \Right6{} N_G(W)/C_G(W) \cong 
	\Aut_G(W), \]
induced by restriction of automorphisms or by the inclusion $N_G(P)\le 
N_G(W)$ is surjective with kernel $(N_G(P)\cap C_G(W))/C_G(P)P$ of odd 
order.

\smallskip

\noindent\textbf{(b) } If $W\in\5\calz(\calf)$ and $P=C_S(W)$, then 
$P\in\syl2{C_G(W)}$ and $W=\Omega_1(Z(P))$ by definition. So we are in the 
situation of (a), and $\Out_G(P)$ has a strongly embedded subgroup 
since $\Aut_G(W)$ does. Also, $N_G(P)\le N_G(W)$, while 
$N_S(P)=N_S(W)\in\syl2{N_G(W)}$ since $W$ is fully normalized in $\calf$. 
Hence $N_S(P)\in\syl2{N_G(P)}$, so $P$ is also fully normalized and 
$P\in\EE_\calf$. 
\end{proof}

Our proof that $\Ker(\mu_G)=1$ in all cases is based on the following 
proposition, which is a modified version of similar results in \cite{AOV1} 
and \cite{BMO2}. In most cases handled in the next section, point (e) 
suffices to prove that $\Ker(\mu_G)=1$.

When $\alpha\in\Aut(\call)$ and $P$ is an object in $\call$, we let 
$\alpha_P\:\Aut_\call(P)\Right2{}\Aut_\call(\alpha(P))$ denote the restriction 
of $\alpha$ to $\Aut_\call(P)$.

\begin{Prop} \label{Ker(mu):AOV1}
Fix a finite group $G$, choose $S\in\syl2{G}$, and set $\calf=\calf_S(G)$ 
and $\call=\call_S^c(G)$. Each element in $\Ker(\mu_G)$ is 
represented by some $\alpha\in\Aut(\call)$ such that 
$\alpha_S=\Id_{\Aut_\call(S)}$.  For each such $\alpha$, there are 
elements $g_P\in{}C_{Z(P)}(\Aut_S(P))=Z(N_S(P))$, defined for each fully 
normalized subgroup $P\in\Ob(\call)$, for which the following hold:
\begin{enuma} 

\item The automorphism $\alpha_P\in\Aut(\Aut_\call(P))$ is conjugation by 
$[g_P]\in\Aut_\call(P)$, and $g_P$ is uniquely determined by $\alpha$ modulo 
$C_{Z(P)}(\autf(P))$. In particular, $\alpha_P=\Id_{\Aut_\call(P)}$ if and 
only if $g_P\in{}C_{Z(P)}(\autf(P))$.  

\item Assume $P,Q\in\Ob(\call)$ are both fully normalized in $\calf$.  
If $Q=\9aP$ for some $a\in{}S$, then we can choose $g_Q=\9ag_P$.  

\item If $Q\le{}P$ are both fully normalized and are objects in $\call$, 
then $g_P\equiv{}g_Q$ modulo $C_{Z(Q)}(N_G(P)\cap N_G(Q))$. 

\item Assume, for each $W\in\5\calz(\calf)$ and $P=C_S(W)$, that $g_P\in 
C_{Z(P)}(\autf(P))$ (equivalently, that $\alpha_P=\Id_{\Aut_\call(P)}$). 
Then $\alpha=\Id$.

\item If $\5\calz(\calf)=\emptyset$, then $\Ker(\mu_G)=1$. If 
$|\5\calz(\calf)|=1$, and $|Z(S)|=2$ or (more generally) 
$\autf(\Omega_1(Z(S)))=1$, then $\Ker(\mu_G)=1$.

\end{enuma}
\end{Prop}

\begin{proof} Points (a)--(c) are part of \cite[Proposition 4.2]{AOV1}, (d) 
follows from \cite[Proposition A.2(d)]{BMO2}, and (e) combines parts (a) 
and (b) in \cite[Proposition A.2]{BMO2}.
\end{proof}

The following notation will be useful in the next lemma, and in the next 
section.

\begin{Defi} 
For each finite group $G$ and each $k\ge0$, let $\IND{k}G$ be the set of 
subgroups $H\le G$ such that $[G:H]=2^k\cdot m$ for some odd $m$. Let 
$\IND{\le k}G$ be the union of the sets $\IND{\ell}G$ for $0\le\ell\le k$.
\end{Defi}

\begin{Lem} \label{H/Q-essential}
Let $H$ be a finite group, fix $T\in\syl2{H}$, and set $\calf=\calf_T(H)$. 
Set $Q=O_2(H)$, and assume $C_H(Q)\le Q$. Assume $W\in\5\calz(\calf)$, and 
set $P=C_T(W)$. Set $V=\Omega_1(Z(Q))$, and set $H^*=\Aut_H(V)$, 
$P^*=\Aut_P(V)$, $T^*=\Aut_T(V)$, and $\calf^*=\calf_{T^*}(H^*)$. 
\begin{enuma} 

\item We have $W\le V$, $\Aut_H(W)=\Aut_{H^*}(W)$ has a strongly 
embedded subgroup, $P^*$ is a radical $2$-subgroup of $H^*$, and 
$N_{H^*}(P^*)/P^*$ has a strongly embedded subgroup.

\item If $H^*$ is a Chevalley group (i.e., untwisted) over the field 
$\F_2$, then $P^*\in\EE_{\calf^*}\subseteq\IND1{H^*}$. If $H^*\cong\SU_{2n}(2)$ 
or $\Omega_{2n}^-(2)$ for $n\ge2$, then 
$P^*\in\EE_{\calf^*}\subseteq\IND{\le2}{H^*}$. 

\item If $H^*\cong A_6$, $A_7$, or $M_{24}$, then 
$P^*\in\EE_{\calf^*}\subseteq\IND1{H^*}$. If $H^*\cong M_{22}$ or $M_{23}$, 
then $P^*\in\EE_{\calf^*}\subseteq\IND{\le2}{H^*}$. If $H^*\cong S_5$, then 
$P^*\in\IND{\le2}{H^*}$. 

\item If $H^*\cong\Aut(M_{22})$, then 
$P^*\in\EE_{\calf^*}\subseteq\IND{\le2}{H^*}$, and $P^*\cap 
O^2(H^*)\in\EEE{O^2(H^*)}$. 

\end{enuma}
\end{Lem}

\begin{proof} Fix $W\in\5\calz(\calf)$, and set $P=C_T(W)$ as above. Then 
$P\in\EE_\calf$ by Lemma \ref{Out->Aut}(b). Also, $W=\Omega_1(Z(P))$ and 
$P\ge O_2(C_H(V))=Q$, and hence $W\le\Omega_1(Z(Q))=V$. 

\smallskip

\textbf{(a) } Since $V\nsg H$, each $\alpha\in\Aut_H(W)$ extends to 
$\4\alpha\in\Aut_H(V)=H^*$, and thus $\Aut_{H^*}(W)=\Aut_H(W)$. Hence 
	\[ N_{H^*}(P^*)/P^* \cong N_{H/Q}(P/Q)\big/(P/Q) 
	\cong N_H(P)/P \cong \Out_H(P), \]
so this group has a strongly embedded subgroup. In particular, 
$P^*=O_2(N_{H^*}(P^*))$ (see \cite[Proposition A.7(c)]{AKO}), so $P^*$ is 
a radical 2-subgroup of $H^*$.

\smallskip

\noindent\textbf{(b) } Since $W\le V$, $W=\Omega_1(Z(P))=C_V(P^*)$. By (a), 
$N_{H^*}(P^*)/P^*$ has a strongly embedded subgroup, and 
$O_2(N_{H^*}(P^*))=P^*$. 

If $H^*$ is a group of Lie type over the field $\F_2$, then by the 
Borel-Tits theorem (see \cite[Corollary 3.1.5]{GLS3}), $N_{H^*}(P^*)$ is a 
parabolic subgroup and $P^*=O_2(N_{H^*}(P^*))$. Hence $P^*\in\EE_{\calf^*}$ 
in this case. Also, $O^{2'}(\Out_{H^*}(P))\cong O^{2'}(H/P)$ is a central 
product of groups of Lie type in characteristic $2$ (cf. \cite[Proposition 
2.6.5(f,g)]{GLS3}). Since it has a strongly embedded subgroup, it must be 
isomorphic to $\SL_2(2)\cong S_3$ (hence $P\in\IND1{H^*}$), or possibly to 
$A_5\cong\SL_2(4)\cong\Omega_4^-(2)$ if $H^*\cong\SU_{2n}(2)$ or 
$\Omega_{2n}^-(2)$ for $n\ge2$ (in which case $P\in\IND2{H^*}$). Note that 
we cannot get $\SU_3(2)$ since we only consider even dimensional 
unitary groups.

\smallskip

\noindent\textbf{(c) } If $H^*\cong M_n$ for $n=22,23,24$, then by 
\cite[pp. 42--44]{GL}, it is of characteristic 2 type, in the sense that 
all 2-local subgroups are strictly 2-constrained. So $N_{H^*}(P^*)$ is 
strictly 2-constrained, $P^*$ is centric in this group, and hence 
$P^*\in\EE_{\calf^*}$. Also, $\EE_{\calf^*}\subseteq\IND1{H^*}$ if 
$H^*\cong M_{24}$ \cite[Proposition 6.5]{OV2}, while 
$\EE_{\calf^*}\subseteq\IND{\le2}{H^*}$ if $H^*\cong M_{22}$ or $M_{23}$ 
\cite[Table 5.2]{OV2}. 

The remaining cases ($H^*\cong A_6$, $A_7$, or $S_5$) are elementary.

\smallskip

\noindent\textbf{(d) } The radical 2-subgroups of $H^*\cong\Aut(M_{22})$ 
are listed in \cite[Table VIII]{Yoshiara-Co2}. There are just three classes 
of such subgroups $Q$ for which $N(Q)/Q$ has a strongly embedded subgroup, 
of which the members of two have index $2$ in a Sylow 2-subgroup and those 
of the third have index $4$. Each of them is essential in $\Aut(M_{22})$, 
and contains with index $2$ an essential 2-subgroup of $M_{22}$. 
\end{proof}

We will need to identify the elements of $\5\calz(\calf)$, when 
$\calf=\calf_S(G)$ for a sporadic group $G$ and $S\in\syl2{G}$. In most 
cases, it will turn out that $\5\calz(\calf)=\{Z_2(S)\}$, which is why we 
need some tools for identifying this subgroup.

\begin{Lem} \label{[S:C(V)]=2}
Let $S$ be a $2$-group, and assume $W\le S$ is elementary abelian. If 
$[S:C_S(W)]=2$, then $W\le Z_2(S)$ and $\rk(W)\le2\cdot\rk(Z(S))$.
\end{Lem}

\begin{proof} Set $Q=C_S(W)$ for short; $Q\nsg S$ since it has 
index $2$. Then $W\le\Omega_1(Z(Q))$, and upon replacing $W$ by 
$\Omega_1(Z(Q))$, we can arrange that $W\nsg S$. 

Fix $x\in S\sminus Q$. Since $x^2\in Q=C_S(W)$, we have $[W,S]=[W,x]\le 
C_W(x)\le Z(S)$. So $W\le Z_2(S)$, and $\rk(W)\le2\cdot\rk(Z(S))$.
\end{proof}

\begin{Lem} \label{Z2(S)}
Fix a finite group $G$ and a Sylow $2$-subgroup $S\in\syl2{G}$. 
\begin{enuma} 

\item If $G$ is one of the sporadic groups $J_4$, $\Co_2$, $\Co_1$, 
$\Suz$, $\Ru$, $\Fi_{24}'$, $F_5$, $F_3$, $F_2$, or $F_1$, then $|Z(S)|=2$ 
and $Z_2(S)\cong E_4$. If $G\cong\Co_2$, then $Z_2(S)$ has type 
\textup{$\2{abb}$}, while in all other cases, the three involutions in 
$Z_2(S)$ lie in the same $G$-conjugacy class. 

\item If $G\cong\Fi_{22}$, then $Z_2(S)\cong E_8$ is of type 
$\2a_2\0b_3\0c_2$ and contains a subgroup of type $\2b^2$. If 
$G\cong\Fi_{23}$, then $Z_2(S)\cong E_{16}$. 

\item If $G\cong \HS$, $\ON$, or $\Co_3$, then $|Z(S)|=2$ and $Z_2(S)\cong 
C_4\times C_2$. 

\end{enuma}
\end{Lem}

\begin{proof} \textbf{(a) } In each of these cases, we choose $Q\nsg S$ and 
$H=N_G(Q)$ as follows, where $H^*=H/Q\cong\Aut_H(Q)$:
	\[ \renewcommand{\arraystretch}{1.5}
	\begin{array}{|c||c|c|c|c|c|c|c|c|c|c|}
	\hline
	G & \Co_1 & \Suz & \Ru & F_5 & F_3 
	& F_2 & F_1 & J_4 & \Co_2 & \Fi_{24}' \\\hline\hline
	Q & 2^{1+8}_+ & 2^{1+6}_- & 2.2^{4+6} & 2^{1+8}_+ & 2^{1+8}_+ 
	& 2^{1+22}_+ & 2^{1+24}_+ & E_{2^{11}} & E_{2^{10}} & E_{2^{11}} \\\hline
	H^* & \Omega_8^+(2) & \Omega_6^-(2) & S_5 & \Omega_4^+(4):2 
	& A_9 & \Co_2 & \Co_1 & M_{24} & M_{22}:2 & M_{24} \\\hline
	\end{array} \]
References for all of these subgroups are given in the next section.

Assume that $|Z(Q)|=2$; i.e., that we are in one of the first seven cases. 
Then $|Z(S)|=2$, and $Z_2(S)\le Q$ since $H^*$ acts faithfully on $Q/Z(Q)$. 
Set $\4Q=Q/Z(Q)$, so that $Z_2(S)/Z(S)=C_{Z(\4Q)}(S/Q)$. If $Q$ is 
extraspecial, then $\rk(Z_2(S))\ge2$: since $\4Q$ has an odd number of 
isotropic points (cf. \cite[Theorem 11.5]{Taylor}), at least one is fixed 
by $S$. 

When $G\cong\Co_1$ or $\Suz$, $\4Q$ is the natural (orthogonal) module for 
$H^*$, so $|C_{\4Q}(S)|=2$ (see \cite[Theorem 6.15]{Curtis-reps} or 
\cite[Theorem 2.8.9]{GLS3}), and hence $Z_2(S)\cong E_4$. 


When $G\cong\Ru$, $\4Q$ is special of type $2^{4+6}$, $Z_2(Q)\cong E_{32}$, 
and $H/Q$ acts on $Z_2(Q)/Z(Q)$ via the natural action of $\SSL_2(4)$ 
\cite[\S\,1.4]{Wilson-Ru}. So $|C_{Z(\4Q)}(S/Q)|=2$ in this case, and 
$Z_2(S)\cong E_4$.

When $G\cong F_5$, a Sylow 2-subgroup of $O^2(H/Q)\cong\Omega_4^+(4)$ acts 
on $\4Q\cong(\F_4)^4$ with 1-dimensional fixed subgroup. This subgroup 
lifts to $V_3<Q$, where $V_3\cong E_8$ and $\Aut_G(V_3)\cong\GL_3(2)$ (see 
\cite[p. 365]{NW-HN}). Thus $[V_3,S]>Z(S)$, so $Z_2(S)<V_3$, and 
$Z_2(S)\cong E_4$.

When $G\cong F_3$, $\4Q$ as an $\F_2A_9$-module satisfies the hypotheses of 
Lemma \ref{l:A9-rk8} by \cite[3.7]{Parrott-Th}, and hence 
$|C_{\4Q}(S/Q)|=2$ by that lemma.

Assume $G\cong F_1$ or $F_2$. Thus $H^*\cong\Co_1$ or $\Co_2$, 
respectively. Set $T=S/Q\in\syl2{H^*}$, and let $V\nsg T$ and 
$K=N_{H^*}(V)$ be such that $K\cong2^{11}.M_{24}$ or $2^{10}.M_{22}:2$ and 
$V=O_2(K)$. By \cite[Lemmas 3.7.b \& 3.8.b]{MStroth}, $|C_{\4Q}(V)|=2$, and 
hence $|C_{\4Q}(S/Q)|=2$. So $Z_2(S)\cong E_4$ in both cases.

In the remaining three cases, $Q$ is elementary abelian. When 
$G\cong\Co_2$, $Q\cong E_{2^{10}}$ is the Golay module (dual Todd module) 
for $H^*\cong M_{22}:2$. Let $K<H^*$ be the hexad subgroup $K\cong2^4:S_6$, 
chosen so that $K>S^*=S/Q$, and set $R=O_2(K)\cong E_{16}$. Set 
$Q_1=C_Q(R)$ and $Q_5=[R,Q]$. By \cite[Lemma 3.3.b]{MStroth}, $\rk(Q_1)=1$, 
$\rk(Q_5)=5$, and $Q_5/Q_1$ is the natural module for $S_6\cong\Sp_4(2)$. 
Hence $Z(S)=C_Q(S^*)=Q_1$, and 
$Z_2(S)/Z(S)=C_{Q/Q_1}(S^*)=C_{Q_5/Q_1}(S^*)$ also has rank $1$. So 
$Z_2(S)\cong E_4$. The two elements in $Z_2(S)\sminus Z(S)$ are 
$S$-conjugate, and do not lie in $\2c$ since $C_G(x)\in\IND3G$ for 
$x\in\2c$ (see \cite[Table II]{Wilson-Co2}). By \cite[Table II]{Wilson-Co2} 
again, each $\2a$-element acts on the Leech lattice with character $-8$, so 
a subgroup of type $\2a^2$ would act fixing only the zero vector, hence 
cannot be in $\Co_2$. Thus $Z_2(S)$ has type $\2{abb}$.

Assume $G\cong\Fi_{24}'$ or $J_4$. In both cases, $Q\cong E_{2^{11}}$ is 
the Todd module for $H^*\cong M_{24}$ (see \cite[34.9]{A-3tr} and 
\cite[Theorem A.4]{Janko}). Let $K<H/Q$ be the sextet subgroup 
$K\cong2^6:3S_6$, chosen so that $K>S^*=S/Q\in\syl2{H^*}$, and set 
$R=O_2(K)\cong E_{64}$. By \cite[Lemma 3.5.b]{MStroth}, there are 
$\F_2K$-submodules $Q_1<Q_7<Q$ of rank $1$ and $7$, respectively, where 
$Q_1=C_Q(R)$ and $Q_2=[R,Q]$, and where $K/R\cong3S_6$ acts on $Q_7/Q_1$ as 
the dual module to $R$. Thus $Z(S)=Q_1$ and 
$Z_2(S)/Z(S)=C_{Q_7/Q_1}(S^*)\cong R/[S^*,R]$. Since $S^*\cong\UT_5(2)$ 
contains only two subgroups of rank $6$, one easily sees that 
$|R/[S^*,R]|=2$, and hence $Z_2(S)\cong E_4$.

In all of the above cases except $\Co_2$, $S$ contains a normal elementary 
abelian subgroup $V$ of rank at least $2$ all of whose involutions lie in 
the same $G$-conjugacy class. We refer to the lists of maximal 2-local 
subgroups in the next section, where we can take $V=V_i=Z(O_2(H_i))$, for 
$i=2$ (when 
$G\cong\Co_1$, $\Suz$, $F_2$, or $F_1$), $i=3$ (when $G\cong J_4$, $\Ru$, 
or $F_5$), or $i=5$ (for $G\cong\Fi_{24}'$ or $F_3$). Since each normal 
subgroup of order at least $4$ contains $Z_2(S)$, the involutions in 
$Z_2(S)$ also lie in the same class.

\smallskip

\noindent\textbf{(b) } When $G\cong\Fi_{22}$ and $S\in\syl2{G}$, 
$Z(S)=\gen{z}$ has order $2$, and 
$H=C_G(z)\cong(2\times2^{1+8}_+):U_4(2):2$. Set $Q=O_2(H)$. Then $O^2(H/Q)$ 
acts faithfully on $\4Q=Q/Z(Q)$ as a $4$-dimensional unitary space over 
$\F_4$, so $\dim_{\F_4}(C_{\4Q}(S\cap O^2(H)))=1$ \cite[Theorem 
6.15]{Curtis-reps}. An involution $hQ$ with 
$h\in H\sminus O^2(H)$ acts as a field automorphism on the unitary space 
$\4Q$, so $\dim_{\F_2}(C_{\4Q}(S))=1$. Since $|Z(Q)|=4$, this proves that 
$|Z_2(S)|\le8$.

To see that $Z_2(S)$ does contain a subgroup of rank $3$, consider a hexad 
group $V\cong E_{32}$ normal in $S$, generated by six transpositions 
$\{a_1,\dots,a_6\}$ (where $a_1\cdots a_6=1$), ordered so that 
$\Aut_S(V)=\gen{(1\,2)(3\,4),(1\,2)(5\,6),(1\,3)(2\,4)}$. Then 
	\[ Z(S)=C_V(S)=\gen{a_5a_6} \qquad\textup{and}\qquad
	Z_2(S)=\gen{a_1a_2,a_3a_4,a_5,a_6} \textup{ is of type } 
	\2a_2\0b_3\0c_2, \]
and $\gen{a_1a_2,a_3a_4}<Z_2(S)$ has type $\2b^2$.

When $G\cong\Fi_{23}$ and $S\in\syl2{G}$, $Z(S)\cong E_4$ contains 
involutions $x,y,z$ in each of the three classes $\2a$, $\2b$, and $\2c$, 
respectively. Also, $C_G(x)\cong2\Fi_{22}$, so we can identify $S/\gen{x}$ 
as a Sylow 2-subgroup of $\Fi_{22}$, whose center lifts to a pair of 
elements of class $\2b$ and $\2c$ in $G$. Thus $S/Z(S)\cong T/Z(T)$ when 
$T\in\syl2{\Fi_{22}}$, we already saw that $|Z(T/Z(T))|=4$, and so 
$|Z_2(S)|=16$. All involutions in $\Fi_{22}$ lift to involutions in 
$2\cdot\Fi_{22}<G$, so $Z_2(S)$ is elementary abelian. 

\smallskip

\noindent\textbf{(c) } When $G\cong\HS$ or $\ON$, this follows from 
the descriptions by Alperin \cite[Corollary 1]{Alperin} and O'Nan 
\cite[\S\,1]{O'Nan} of $S$ as being contained in an extension of the form 
$4^3.L_3(2)$. (In terms of their presentations, $Z(S)=\gen{v_1^2v_3^2}$, 
while $Z_2(S)=\gen{v_1v_3,v_1^2v_2^2}$.) When $G\cong\Co_3$, it follows 
from a similar presentation of $S\le4^3.(2\times L_3(2))$ (see, e.g., 
\cite[\S\,7]{OV2}). 
\end{proof}

\bigskip

\section{Injectivity of $\mu_G$} 
\label{s:Ker(mu)=1}


We are now ready to prove, when $p=2$, that $\Ker(\mu_G)=1$ for each of the 
sporadic groups $G$ not handled in Proposition \ref{kappa-2-small}. This 
will be done in each case by determining the set $\5\calz(\calf)$ and then 
applying Proposition \ref{Ker(mu):AOV1}. One can determine $\5\calz(\calf)$ 
using the lists of radical 2-subgroups found in \cite{Yoshiara-Co2} and 
other papers. However, we decided to do this instead using lists of maximal 
2-local subgroups, to emphasize that the details needed to prove this 
result are only a small part of what is needed to determine the radical 
subgroups.

\begin{Prop} \label{mu-2-big}
Assume $p=2$, and let $G$ be a sporadic simple group whose Sylow 
$2$-subgroups have order at least $2^{10}$. Then $\Ker(\mu_G)=1$.
\end{Prop}

\begin{proof} There are fifteen groups to consider, and we go through the 
list one or two at a time. In each case, we fix $S\in\syl2{G}$ and set 
$\calf=\calf_S(G)$, $\call=\call_S^c(G)$, and $\5\calz=\5\calz(\calf)$. 
When we list representatives for the conjugacy classes of maximal 2-local 
subgroups of $G$, we always choose them so that each such $H$ satisfies 
$S\cap H\in\syl2{H}$. In particular, if $H$ has odd index in $G$, then 
$H\ge S$ and hence $O_2(H)\nsg S$ and $Z(O_2(H))\nsg S$ (making the choice 
of $H$ unique in most cases).

In four of the cases, when $G\cong M_{24}$, $\He$, $\Co_2$, or $\Fi_{23}$, 
$\5\calz$ has two members, and we use Proposition \ref{Ker(mu):AOV1}(b,c,d) 
to prove that $\mu_G$ is injective. In all of the other cases, 
$|\5\calz|=1$ and $|Z(S)|=2$, and we can apply Proposition 
\ref{Ker(mu):AOV1}(e). 
Recall that by Proposition \ref{Ker(mu):AOV1}, each class in $\Ker(\mu_G)$ 
contains an element $\alpha\in\Aut(\call)$ which acts as the identity on 
$\Aut_\call(S)$.

Note that whenever $|Z(S)|=2$ and $W\cong E_4$ 
is normal in $S$, $[S:C_S(W)]=2$, and hence $W\le Z_2(S)$ by Lemma 
\ref{[S:C(V)]=2}.

For convenience, we sometimes write $A\cj[H]B$ to mean that $A$ is 
$H$-conjugate to $B$, and $A\cjle[H]B$ to mean that $A$ is $H$-conjugate to 
a subgroup of $B$.

\bigskip

\begin{spor}{$M_{24}$, $\He$} We identify $S$ with $\UT_5(2)$, the group of 
$(5\times5)$ upper triangular matrices over $F_2$. Let $e_{ij}\in S$ (for 
$i<j$) be the matrix with $1$'s on the diagonal, and with unique nonzero 
off-diagonal entry $1$ in position $(i,j)$. Set $W_1=\gen{e_{15},e_{25}}$ 
and $W_4=\gen{e_{14},e_{15}}$, $Q_i=C_S(W_i)$ for $i=1,4$, and 
$Q_{14}=Q_1\cap Q_4$. By \cite[Propositions 6.2 \& 6.9]{OV2}, $Q_1$ and 
$Q_4$ are essential in $G$, and are the only essential subgroups with 
noncyclic center. Hence by Lemma \ref{Out->Aut}, $\5\calz=\{W_1,W_4\}$. 
Also, $Q_{14}=A_1A_2$, where $A_1$ and $A_2$ are the unique subgroups of 
$S$ of type $E_{64}$, and hence $Q_{14}=J(S)$ is characteristic in $S$, 
$Q_1$, and $Q_4$.

Fix $\alpha\in\Aut(\call)$ which is the identity on $\Aut_\call(S)$. By 
Proposition \ref{Ker(mu):AOV1}(a), there are elements $g_P\in 
C_{Z(P)}(\Aut_S(P))$, chosen for each $P\le S$ which is fully normalized in 
$\calf$ and $2$-centric in $G$, such that $\alpha|_{\Aut_\call(P)}$ is 
conjugation by $[g_P]$. Then $g_{Q_1}=g_{Q_{14}}=g_{Q_4}\in Z(S)$ by point 
(c) in the proposition, since for $i=1,4$, $C_{Z(Q_{14})}(N_G(Q_i))=1$. Set 
$g=g_{Q_1}$; upon replacing $\alpha$ by $c_g^{-1}\circ\alpha$, we can 
arrange that $\alpha|_{\Aut_\call(Q_i)}=\Id$ for $i=1,4$ without changing 
$\alpha|_{\Aut_\call(S)}$. Hence $\Ker(\mu_G)=1$ by Proposition 
\ref{Ker(mu):AOV1}(d).

\end{spor}

\begin{spor}{$J_4$} By \cite[\S\,2]{KW-J4}, there are four conjugacy classes 
of maximal 2-local subgroups, represented by: 
	\[ H_1\cong2^{1+12}_+.3M_{22}:2,\quad 
	H_3\cong2^{3+12}.(\Sigma_5\times L_3(2)), \quad
	H_{10}\cong2^{10}:L_5(2), \quad
	H_{11}\cong2^{11}:M_{24}. \]
Set $Q_i=O_2(H_i)$ and $V_i=Z(Q_i)\cong E_{2^i}$. Note that 
$H_{10}\in\IND1G$, while $H_i\ge S$ for $i\ne10$.

Fix $W\in\5\calz$ and set $P=C_S(W)$. Then $N_G(W)\cjle H_i$ for some $i$, 
in which case $P\cjge Q_i$ and $W\cjle V_i$ by Lemma 
\ref{H/Q-essential}(a). Thus $i>1$ since $\rk(W)\ge2$. By Lemma 
\ref{H/Q-essential}(b,c), 
$\Aut_P(V_i)\in\EEE{\Aut_{H_i}(V_i)}\subseteq\IND1{\Aut_{H_i}(V_i)}$, and 
hence $P\in\IND1{H_i}$. 

Thus either $[S:P]=2$, in which case $W=Z_2(S)\cong E_4$ by Lemmas 
\ref{[S:C(V)]=2} and \ref{Z2(S)}(a); or $i=10$ and $[S:P]=4$. In the latter 
case, since $H_{10}/V_{10}\cong L_5(2)$ acts on $V_{10}$ as 
$\Lambda^2(\F_2^5)$, we have $\rk(W)=\rk(C_{V_{10}}(P/V_{10}))\le
\rk\bigl(C_{V_{10}}([S^*,S^*])\bigr)=2$ for $S^*\in\syl2{H_{10}/V_{10}}$. 
So $W\cong E_4$ in all cases.

By \cite[Table 1]{KW-J4}, there are two classes of four-groups in $G$ whose 
centralizer has order a multiple of $2^{19}$, denoted $AAA^{(1)}$ and 
$ABB^{(1)}$, with centralizers of order $2^{20}\cdot3\cdot5$ and 
$2^{19}\cdot3\cdot5$, respectively. Thus $AAA^{(1)}\cj Z_2(S)$ (Lemma 
\ref{[S:C(V)]=2}), and $W$ lies in one of the two classes. Since 
$\Aut_G(ABB^{(1)})$ is a 2-group, $W\not\cj ABB^{(1)}$. Hence 
$\5\calz=\{Z_2(S)\}$, and $\mu_G$ is injective by Proposition 
\ref{Ker(mu):AOV1}(e).

\end{spor}

\begin{spor}{$\Co_3$} By \cite[Proposition 7.3]{OV2}, there is at 
most one essential subgroup with noncyclic center (denoted $R_1$); and 
$R_1\in\EE_\calf$ since otherwise $N_G(Z(S))$ would control fusion in $G$. 
Also, $\Out_G(R_1)\cong S_3$ and $Z(R_1)\in\5\calz$ by \cite[Propsition 
7.5]{OV2}. So $|\5\calz|=1$, and $\Ker(\mu_G)=1$ by Proposition 
\ref{Ker(mu):AOV1}(e). (In fact, it is not hard to see that 
$\5\calz=\{\Omega_1(Z_2(S))\}$.)

\end{spor}

\begin{spor}{$\Co_2$} By \cite[pp. 113--114]{Wilson-Co2}, each $2$-local 
subgroup of $G$ is contained up to conjugacy in one of the following 
subgroups:
	\[ H_1\cong2^{1+8}_+.\Sp_6(2), \quad H_4\cong2^{4+10}.(S_3\times 
	S_5), \quad H_5\cong(2^4\times2^{1+6}_+).A_8, \quad 
	H_{10}\cong2^{10}:M_{22}:2 \]
	\[ K_1\cong U_6(2):2,\quad K_2\cong \McL,\quad K_3\cong M_{23}. \]
For $i=1,4,5,10$, set $Q_i=O_2(H_i)$ and $V_i=Z(Q_i)\cong E_{2^i}$. 

Recall (Lemma \ref{Z2(S)}(a)) that $Z_2(S)$ has type $\2{abb}$. Set 
$Z_2(S)=\{1,x,y_1,y_2\}$, where $x\in\2a$ and $y_1,y_2\in\2b$. Thus 
$Z(S)=\gen{x}$, $H_1=C_G(x)$, and we can assume $H_5=C_G(y_1)$.

Fix $W\in\5\calz$, and set $P=C_S(W)$. Then $W\ge Z(S)$, so 
$W\cap\2a\ne\emptyset$. If $\rk(W)=2$, then $W$ must have type $\2a^2$. 
Since each $\2a$-element acts on the Leech lattice with character $-8$ 
\cite[Table II]{Wilson-Co2}, $W$ would fix only the zero element, and hence 
cannot be contained in $\Co_2$. Thus $\rk(W)\ge3$. If $[S:P]=2$, then $W\le 
Z_2(S)$ by Lemma \ref{[S:C(V)]=2}, which is impossible since 
$\rk(Z_2(S))=2$. So $[S:P]\ge4$.


If $N_G(W)\cjle K_2\cong\McL$ or $N_G(W)\cjle K_3\cong M_{23}$, then by the 
list of essential subgroups in these groups in \cite[Table 5.2]{OV2}, 
$\rk(W)=\rk(Z(P))\le2$. So these cases are impossible. 

The subgroup $K_1\cong U_6(2):2$ in $\Co_2$ is the stabilizer of a triple 
of 2-vectors in the Leech lattice \cite[pp. 561--2]{Curtis-.0}, which we 
can choose to have the form $(4,4,0,\dots)$, $(0,-4,4,\dots)$, and 
$(-4,0,-4,\dots)$. Using this, we see that the maximal parabolic subgroups 
$2^{1+8}_+:U_4(2):2$, $2^9:L_3(4):2$, and $2^{4+8}:(S_3\times S_5)$ in 
$K_1$ can be chosen to be contained in $H_1$, $H_{10}$, and $H_4$, 
respectively. If $N_G(W)\cjle K_1$, then it is contained in one of the 
maximal parabolics by the Borel-Tits theorem, and so $N_G(W)$ is also 
conjugate to a subgroup of one of the $H_i$.

Thus in all cases, we can assume that $N_G(W)\le H_i$ for some 
$i=1,4,5,10$. Then $P\ge Q_i$ and $W=Z(P)\le V_i$, so $i\ne1$.

Assume $i=5$, and recall that $\Fr(Q_5)=\gen{y_1}$. The image of $W$ in 
$V_5/\gen{y_1}\cong E_{16}$ has rank at least $2$ since $\rk(W)\ge3$, so 
$\Aut_{N_G(W)}(V_5/\gen{y_1})$ is the stabilizer subgroup of a projective 
line and plane in $A_8\cong\SL_4(2)$ (a line and plane determined by $S$). 
So there is at most one member of $\5\calz$ whose normalizer is in 
$H_5=C_G(y_1)$, and it has rank $3$ if it exists.

Now, $V_4\ge Z_2(S)$ since it is normal in $S$. Since $Z_2(S)$ has type 
$\2{abb}$, $S_5$ must act on $V_4^\#$ with orbits of order 
$5$ and $10$, and has type $\2a_5\0b_{10}$. So if $i=4$, then $W$ is a rank 
3 subgroup of the form $\2a_3\0b_4$ (the centralizer of a 2-cycle in 
$S_5$). There is exactly one $\2b$-element in $W$ whose product with each 
of the other $\2b$-elements is in class $\2a$, so $N_G(W)\cjle 
H_5=N_G(\2b)$: a case which we have already handled. 

Assume $i=10$, and set 
$H^*=H_{10}/V_{10}\cong\Aut_{H_{10}}(V_{10})\cong\Aut(M_{22})$ and 
$P^*=P/V_{10}$. By Lemma \ref{H/Q-essential}(d), $P^*\cap O^2(H^*)$ is an 
essential 2-subgroup of $O^2(H^*)\cong M_{22}$. Since 
$P\notin\IND1{H_{10}}$, $P^*\cap O^2(H^*)$ has the form $2^4:2<2^4:S_5$ 
(the duad subgroup) by \cite[Table 5.2]{OV2}, and this extends to 
$P^*\cong2^5:2<2^5:S_5<\Aut(M_{22})$. But $V_{10}.2^5$ has center $V_4$ 
(see \cite[Lemma 3.3]{MStroth}), and so we are back in the case $i=4$.

Thus $\5\calz=\{W_1,W_2\}$, where $\rk(W_i)=3$ and $N_G(W_i)\le C_G(y_i)\cj 
H_5$ for $i=1,2$. (These also correspond to the two 2-cycles in 
$\Aut_S(V_4)<S_5$.) Set $P_i=C_S(W_i)$. Fix $\alpha\in\Aut(\call)$ which is 
the identity on $\Aut_\call(S)$, and let $g_i=g_{P_i}\in 
C_{W_i}(\Aut_S(P_i))=Z_2(S)$ ($i=1,2$) be as in Proposition 
\ref{Ker(mu):AOV1}. Thus $\alpha|_{\Aut_\call(P_i)}$ is conjugation by 
$g_i$. Since $y_i\in Z(N_G(P_i))$, we can replace $g_i$ by $g_iy_i$ if 
necessary and arrange that $g_i\in Z(S)$. Then $g_1=g_2$ by Proposition 
\ref{Ker(mu):AOV1}(b) and since $P_1$ and $P_2$ are $S$-conjugate. Upon 
replacing $\alpha$ by $c_{g_1}^{-1}\circ\alpha$, we can arrange that 
$\alpha|_{\Aut_\call(Q_i)}=\Id$ for $i=1,4$ without changing 
$\alpha|_{\Aut_\call(S)}$. Hence $\Ker(\mu_G)=1$ by Proposition 
\ref{Ker(mu):AOV1}(d).

\end{spor}

\begin{spor}{$\Co_1$} There are three conjugacy classes of involutions in 
$G$, of which those in $\2a$ are 2-central. By \cite[Theorem 
2.1]{Curtis-Co1}, each $2$-local subgroup of $G$ is contained up to 
conjugacy in one of the subgroups
	\[ H_1\cong2^{1+8}_+.\Omega_8^+(2), \quad 
	H_2\cong2^{2+12}.(A_8\times S_3), \quad
	H_4\cong2^{4+12}.(S_3\times3S_6), \quad 
	H_{11}\cong2^{11}M_{24}; \]
	\[ K_1\cong(A_4\times G_2(4)):2, \quad 
	K_2\cong(A_6\times U_3(3)):2. \]
Curtis also included $\Co_2$ in his list, but it is not needed, 
as explained in \cite[p. 112]{Wilson-Co2}. Set $Q_i=O_2(H_i)$ and 
$V_i=Z(Q_i)\cong E_{2^i}$. 

Assume $W\in\5\calz$. Then $W\ge Z(S)$, so $W\cap\2a\ne\emptyset$. 
If $W\cap\2c\ne\emptyset$, then $N_G(W)\le H_i$ for some $i=1,2,4,11$ by 
\cite[Lemma 2.2]{Curtis-Co1} (where the involution centralizer in the 
statement is for an involution of type $\2a$ or $\2c$). If $W$ contains 
no $\2c$-elements, then by the argument given in \cite[p. 417]{Curtis-Co1}, 
based on the action of the elements on the Leech lattice, a product of 
distinct $\2a$-elements in $W$ must be of type $\2a$. So in this case, 
$\gen{W\cap\2a}$ is $\2a$-pure, and its normalizer is contained in some 
$H_i$ by \cite[Lemma 2.5]{Curtis-Co1} (together with Wilson's remark 
\cite[p. 112]{Wilson-Co2}).

Set $P=C_S(W)$; then $P\ge Q_i$ and hence $W\le V_i$. Also, $i\ne1$ since 
$\rk(W)>1$. By Lemmas \ref{H/Q-essential}(b,c), 
$\Aut_P(V_i)\in\EEE{\Aut_{H_i}(V_i)}\subseteq\IND1{\Aut_{H_i}(V_i)}$. Since 
$H_i\ge S$, we have $[S:P]=2$, and $W=\Omega_1(Z(P))\le Z_2(S)$ by Lemma 
\ref{[S:C(V)]=2}, with equality since $|Z_2(S)|=4$ by Lemma \ref{Z2(S)}. It 
follows that $\5\calz=\{Z_2(S)\}$, and $\Ker(\mu_G)=1$ by Proposition 
\ref{Ker(mu):AOV1}(e).


\end{spor}

\begin{spor}{$\Suz$} By \cite[\S\,2.4]{Wilson-Suz}, there are three classes of 
maximal 2-local subgroups which are normalizers of $\2a$-pure subgroups, 
represented by 
	\[ H_1\cong2^{1+6}_-.\Omega_6^-(2), \qquad
	H_2\cong2^{2+8}.(A_5\times S_3), \qquad
	H_4\cong2^{4+6}.3A_6. \] 

Fix $W\in\5\calz$, and set $P=C_S(W)$. Since $W\ge Z(S)$, it contains 
$\2a$-elements, and since $\gen{W\cap\2a}$ is $\2a$-pure by \cite[p. 
165]{Wilson-Suz}, $N_G(W)\le H_i$ for some $i\in\{1,2,4\}$. Then $P\ge 
O_2(H_i)$ and $W\le V_i\defeq Z(O_2(H_i))$ by Lemma \ref{H/Q-essential}(a), 
so $i\ne1$ since $\rk(W)\ge2$. Hence $i=2$ or $4$, so $\Aut_G(V_i)\cong 
S_3$ or $A_6$, and 
$\Aut_P(V_i)\in\EEE{\Aut_G(V_i)}\subseteq\IND1{\Aut_G(V_i)}$ by Lemma 
\ref{H/Q-essential}(b,c). So $[S:P]=2$, and $W\le Z_2(S)$ by Lemma 
\ref{[S:C(V)]=2}, with equality since $|Z_2(S)|=4$ by Lemma \ref{Z2(S)}. 
Thus $\5\calz=\{Z_2(S)\}$, and $\Ker(\mu_G)=1$ by Proposition 
\ref{Ker(mu):AOV1}(e).

\end{spor}

\begin{spor}{$\Ru$} 
There are two conjugacy classes of involutions, of which the $\2a$-elements 
are 2-central. By \cite[\S\,2.5]{Wilson-Ru}, the normalizer of each 
$\2a$-pure subgroup is contained up to conjugacy in one of the following 
subgroups:
	\[ H_1\cong 2.2^{4+6}.S_5 \qquad 
	H_3\cong 2^{3+8}.L_3(2) \qquad H_6\cong 2^6.G_2(2). \]
Set $Q_i=O_2(H_i)$ and $V_i=Z(Q_i)$. For each $i=1,3,6$, $V_i$ is 
elementary abelian of rank $i$ and $\2a$-pure. 

Fix $W\in\5\calz$, and set $P=C_S(W)\in\EE_\calf$. Then $W\ge Z(S)$, so $W$ 
contains $\2a$-elements. Since the subgroup $W_0=\gen{W\cap\2a}$ is 
$\2a$-pure \cite[p. 550]{Wilson-Ru}, $N_G(W)\le N_G(W_0)\le H_i$ for 
$i\in\{1,3,6\}$. Since $H_i$ is 2-constrained, $P\ge Q_i=O_2(H_i)$ and 
$W\le V_i$ by Lemma \ref{H/Q-essential}(a). Hence $i\ne1$, since 
$\rk(W)\ge2$.

For $i=3,6$, $\Aut_G(V_i)$ is a Chevalley group over $\F_2$, so by Lemma 
\ref{H/Q-essential}(b), $\Aut_P(V_i)\in\IND1{\Aut_G(V_i)}$, and hence 
$P\in\IND1{H_i}$. So $|P|=2^{13}$ (if $i=3$) or $2^{11}$ (if $i=6$). Also, 
$W$ is $\2a$-pure since $V_i$ is. By \cite[\S\,2.4]{Wilson-Ru}, there are 
four classes of subgroups of type $\2a^2$, of which only one has 
centralizer of order a multiple of $2^{11}$, and that one must be the class 
of $Z_2(S)$ (Lemma \ref{Z2(S)}). So $W=Z_2(S)$ if $i=3$, or if $i=6$ and 
$\rk(W)=2$. 

As explained in \cite[\S\,2.5]{Wilson-Ru}, if $W\le V_6$ and $\rk(W)\ge3$, 
then either $N_G(W)\cjle H_1$, or $N_G(W)$ is in the normalizer of a group 
of the form $\2a^2$ which must be conjugate to $Z_2(S)$ by the above 
remarks, or $C_G(W)=V_6$. The first case was already handled. If 
$N_G(W)\cjle N_G(Z_2(S))$, then $N_G(W)\cjle H_3$ by \cite[p. 
550]{Wilson-Ru}, and this case was already handled. If $C_G(W)=V_6$, then 
$W=P=V_6$, which is impossible since $G_2(2)$ does not have a strongly 
embedded subgroup. Thus $\5\calz=\{Z_2(S)\}$, and $\mu_G$ is injective by 
Proposition \ref{Ker(mu):AOV1}(e).

\end{spor}

\begin{spor}{$\Fi_{22}$, $\Fi_{23}$, or $\Fi_{24}'$} It will be simplest to 
handle these three groups together. Their maximal $2$-local subgroups were 
determined in \cite[Proposition 4.4]{Wilson-Fi22}, \cite{Flaass}, and 
\cite[Theorem D]{Wilson-Fi}, and are listed in Table \ref{tbl:Fischer}. To 
make it clearer how 2-local subgroups of one Fischer group lift to larger 
ones, we include the maximal 2-local subgroups in $\Fi_{21}\cong\PSU_6(2)$ 
(the maximal parabolic subgroups by the Borel-Tits theorem), and give the 
normalizers in $\Fi_{24}$ of the maximal 2-local subgroups of $\Fi_{24}'$.
\begin{table}[ht] 
\renewcommand{\arraystretch}{1.3}
\[ \begin{array}{|c||c|c|c|c|}
\hline
  & \PSU_6(2)=\Fi_{21} & \Fi_{22} & \Fi_{23} & \Fi_{24} \\\hline
K_1 &&& 2\cdot\Fi_{22} & (2\times2\cdot\Fi_{22}).2 \\\hline
K_2 && 2\cdot\Fi_{21} & 2^2\cdot\Fi_{21}.2 & (2\times2^2\cdot\Fi_{21}).S_3 \\\hline
K_3 &&& S_4\times\Sp_6(2) & S_4\times\Omega_8^+(2):S_3 \\\hline
H_1 & 2^{1+8}_+:U_4(2) & (2{\times}2^{1+8}_+:U_4(2)).2 & 
(2^2{\times}2^{1+8}_+).(3{\times} U_4(2)).2 & 
(2^{1+12}_+).3U_4(3).2^2 \\\hline
H_2 & 2^{4+8}:(A_5\times S_3) & 2^{5+8}:(A_6\times S_3) & 
2^{6+8}:(A_7\times S_3) & 2^{7+8}:(A_8\times S_3) \\\hline
H_3 & 2^9:M_{21} & 2^{10}.M_{22} & 2^{11}.M_{23} & 2^{12}.M_{24} \\\hline
H_4 && 2^6.\Sp_6(2) & [2^7.\Sp_6(2)] & 2^8:\SO_8^-(2) \\\hline
H_5 &&&& 2^{3+12}(\SL_3(2){\times} S_6) \\\hline
\end{array} \]
\caption{}
\label{tbl:Fischer}
\end{table}
Also, we include one subgroup which is not maximal: $H_4\le\Fi_{23}$ is 
contained in $K_1$. 

As usual, set $Q_i=O_2(H_i)$ and $V_i=Z(Q_i)$. For each of the four groups 
$\Fi_n$, $H_i\ge S$ for $i=1,2,3,5$. We write $K_i^{(n)}$, $H_i^{(n)}$, 
$Q_i^{(n)}$, or $V_i^{(n)}$  when we need to distinguish $K_i$, $H_i$, 
$Q_i$, or $V_i$ as a subgroup of $\Fi_n$.

Each of the groups $\Fi_n$ for $21\le n\le24$ is generated by a 
conjugacy class of 3-transpositions. 
By \cite[37.4]{A-3tr}, for $22\le n\le24$, $\Fi_n$ has classes of 
involutions $\calj_m$, for $m=1,2,3$ when $n=22,23$ and for $1\le m\le4$ 
when $n=24$. Each member of $\calj_m$ is a product of $m$ commuting 
transpositions (its \emph{factors}): a unique such product except when 
$n=22$ and $m=3$ (in which case each $x\in\calj_3$ has exactly two sets of 
factors) and when $n=24$ and $m=4$. Note that $\calj_1=\2a$, $\calj_2=\2b$, 
and $\calj_3=\2c$ in $\Fi_{22}$ and $\Fi_{23}$, while $\calj_2=\2a$ and 
$\calj_4=\2b$ in $\Fi_{24}'$ (and the other two classes are outer 
automorphisms).

In all cases, $K_1$, $K_2$, and $H_1$ are normalizers of sets of 
$(n-22)$, $(n-21)$, and $(n-20)$ pairwise commuting transpositions. Also, 
$H_3$ is the normalizer of the set of all $n$ transpositions in $S$; these 
generate $Q_3=V_3$ of rank $n-12$, and form a Steiner system of type 
$(n-19,n-16,n)$. Then $H_2$ is the normalizer of a pentad, hexad, heptad, 
or octad of transpositions: one of the members in that Steiner system. From 
these descriptions, one sees, for example, that a subgroup of type $K_i$ 
($i=1,2$) or $H_i$ ($i=1,2,3$) in $\Fi_{22}$ lifts to a subgroup of type 
$K_i$ or $H_i$, respectively, in $2\cdot\Fi_{22}<\Fi_{23}$ and in 
$2\Fi_{22}.2<\Fi_{24}'$.

By \cite[Lemma 4.2]{Wilson-Fi22}, each $\2b$-pure elementary abelian 
subgroup of $\Fi_{22}$ ($\2b=\calj_2$) supports a symplectic form for which 
$(x,y)=1$ exactly when conjugation by $y$ exchanges the two factors of $x$. 
Then $V_4^{(22)}$ is characterized as a subgroup of type 
$\2b^6$ with nonsingular symplectic form. Since each $\2b$-element in 
$\Fi_{22}$ lifts to a $\2b$- and a $\2c$-element in 
$2\cdot\Fi_{22}<\Fi_{23}$, $H_4^{(22)}$ lifts to $H_4^{(23)}$ of the form 
$2^7.\Sp_6(2)$. 

By \cite[Corollary 3.2.3]{Wilson-Fi}, each elementary abelian subgroup of 
$G\cong\Fi_{24}'$ supports a symplectic form where $(x,y)=1$ if and only if $y$ 
is in the ``outer half'' of $C_G(x)\cong2\cdot\Fi_{22}.2$ or 
$2^{1+12}_+.3U_4(3):2$. By \cite[Proposition 3.3.3]{Wilson-Fi}, the form on 
$V_4^{(24)}\cong E_{2^8}$ is nonsingular, and $V_4^{(24)}$ contains 
elements in both classes $\2a=\calj_2$ and $\2b=\calj_4$. If $x\in 
V_4\cap\2a$, then $V_4\cap O^2(C_G(x))/\gen{x}$ has rank $6$ with 
nonsingular symplectic form in $\Fi_{22}$, and hence $\bigl(C_{H_4}(x)\cap 
O^2(C_G(x))\bigr)\big/\gen{x}$ is conjugate to $H_4^{(22)}$. Thus 
$H_4^{(24)}$ contains a lifting of $H_4^{(22)}$ via the inclusion 
$2\cdot\Fi_{22}<\Fi_{24}'$.

Fix $W\in\5\calz$, and set $P=C_S(W)$. If $N_G(W)\le K_i$ for $i=1$ or $2$, 
then since $W\ge Z(S)$, and $O_2(K_i)$ does not contain involutions of all 
classes represented in $Z(S)$ (note that $O_2(K_i^{(24)})\cap\Fi_{24}'$ is 
$\2a$-pure for $i=1,2$), we have $\4W=(W\cap F^*(K_i))\big/O_2(K_i)\ne1$. 
Thus $N_G(\4W)$ is a 2-local subgroup of $F^*(K_i)/O_2(K_i)\cong\Fi_{22}$ 
or $\Fi_{21}$, and hence is contained up to conjugacy in one of its maximal 
2-local subgroups. So (after applying this reduction twice if $i=1$), 
$N_G(W)\le H_i$ for some $1\le i\le4$. We will see below that we can also 
avoid the case $N_G(W)\le K_3$ (when $G\cong\Fi_{23}$ or $\Fi_{24}'$), and 
hence that in all cases, $N_G(W)\cjle H_i$ for some $1\le i\le5$.

\boldd{When $G\cong\Fi_{22}$,} we just showed that (up to conjugacy) we can 
assume $N_G(W)\le H_i$ for some $i=1,2,3,4$. If $i=4$, then by Lemma 
\ref{H/Q-essential}(b), $W=C_{V_4}(P/V_4)$ where $P/V_4\in\EEE{H_4/V_4}$ 
and $H_4/V_4\cong\Sp_6(2)$, so $W$ must be totally isotropic with respect 
to the symplectic form on $V_4$ described above. But in that case, by 
\cite[Lemma 3.1]{Wilson-Fi22}, the subgroup $W^*>W$ generated by all 
factors of involutions in $W$ is again elementary abelian, and $N_G(W)\le 
N_G(W^*)\le H_j$ for some $j=1,2,3$. 

Thus $N_G(W)\le H_i$ where $i\in\{1,2,3\}$, $H_i$ is 2-constrained, and so 
$P=C_S(W)\ge O_2(H_i)$ and $W=\Omega_1(Z(P))\le V_i$. Also, $i\ne1$ since 
$V_1$ has type $\2{aab}$ (so $\Aut_G(V_1)$ is a 2-group). Hence 
$i=2,3$, and $H_i\in\IND0G$. By Lemma \ref{H/Q-essential}(c), 
$\Aut_P(V_i)\in\EEE{\Aut_{H_i}(V_i)}$, and either $[S:P]=2$, or $i=3$ and 
$[S:P]=4$. In this last case, $P/V_3\cong2^4:2$ is contained in a duad 
subgroup $D\cong2^4:S_5$ in $M_{22}$. Also, $O_2(D)\cong E_{16}$ permutes 
$V_3\cap\2a$ in five orbits of length $4$, each of which forms a hexad 
together with the remaining two transpositions. Hence $C_{V_3}(O_2(D))$ has 
type $\2{aab}$, and cannot contain $W$. 

Thus $[S:P]=2$, and hence $\rk(W)=2$ and $W\le Z_2(S)$ by Lemma 
\ref{[S:C(V)]=2}. By Lemma \ref{Z2(S)}(b), $Z_2(S)$ has rank $3$ and type 
$\2a_2\0b_3\0c_2$. Since $\Aut_G(W)$ is not a 2-group, $W$ must be the 
$\2b$-pure subgroup of rank 2 in $Z_2(S)$. (Note that the factors of the 
involutions in $W$ form a hexad.) Thus $|\5\calz|=1$, and 
$\Ker(\mu_G)=1$ by Proposition \ref{Ker(mu):AOV1}(e).

\boldd{When $G\cong\Fi_{23}$,} $W=\Omega_1(Z(P))$ strictly contains $Z(S)$. 
Hence $\rk(W)\ge3$, and $W$ contains involutions of each type $\2a$, $\2b$, 
and $\2c$. If $|W\cap\2a|=1$ $2$, or $3$, then $N_G(W)\le K_1$, $K_2$, or 
$H_1$, respectively, while if $|W\cap\2a|\ge4$, then $N_G(W)\le H_2$ or 
$H_3$, depending on whether or not the transpositions in $W$ are contained 
in a heptad. So by the above remarks, we can assume in all cases that 
$N_G(W)\le H_i$ for some $i=1,2,3,4$. Since $H_i$ is strictly 
2-constrained, $P\ge Q_i$ and $W\le V_i$. If $i=1$, then $W=V_1$ since it 
has rank at least $3$, and thus $W$ has type $\2a_3\0b_3\0c$. The case 
$i=4$ can be eliminated in the same way as it was when $G\cong\Fi_{22}$.

Assume $N_G(W)\le H_2$ and $W\le V_2$, where $\Aut_G(V_2)\cong A_7$. Write 
$V_2\cap\2a=\{a_1,\dots,a_7\}$, permuted by $\Aut_G(V_2)\cong A_7$ in the 
canonical way. Then (up to choice of indexing), $\Aut_P(V_2)$ is one of the 
two essential subgroups $P_1^*=\gen{(1\,2)(3\,4),(1\,2)(5\,6)}$ and 
$P_2^*=\gen{(1\,2)(3\,4),(1\,3)(2\,4)}$. Set $W_j=C_{V_2}(P_j^*)$ and 
$P_j=C_S(W_j)$; thus $P_j^*=\Aut_{P_j}(V_2)$ and hence $[S:P_j]=2$. Also, 
$W_1=\gen{a_1a_2,a_3a_4,a_5a_6}$ has type $\2a\0b_3\0c_3$, and 
$W_2=\gen{a_5,a_6,a_7}$ has type $\2a_3\0b_3\0c$ (thus $W_2\cj V_1$). 

If $N_G(W)\le H_3$ and $W\le V_3$, then $\Aut_G(V_3)\cong M_{23}$ has three 
essential subgroups, of which two are contained in the heptad group 
$2^4:A_7$ and one in the triad group $2^4:(3\times A_5):2$. In the first 
case, the subgroup $2^4$ acts on $V_3\cap\2a$ fixing a heptad, and we are 
back in the case $N_G(W)\le H_2$. In the second case, the subgroup $2^4$ 
fixes a rank 3 subgroup in $V_3$ generated by three tranpositions, and so 
the essential subgroup $2^4:2$ fixes only $Z(S)$.

Thus $\5\calz=\{W_1,W_2\}$, where $W_1,W_2\le Z_2(S)$ by Lemma 
\ref{[S:C(V)]=2}, and $W_1,W_2<V_2$. Also, $\sigma_2=(5\,6\,7)$ 
normalizes $P_2$ and $Q_2$ and permutes the three $\2b$-elements in $W_1$ 
cyclically, while $\sigma_1=(1\,3\,5)(2\,4\,6)$ normalizes $P_1$ and $Q_2$ 
and permutes the three $\2a$-elements in $W_2$ cyclically. 

Fix $\alpha\in\Aut(\call)$ which is the identity on $\Aut_\call(S)$. 
Let $g_P\in C_{Z(P)}(\Aut_S(P))$, for all $P\in\Ob(\call)$ fully normalized 
in $\calf$, be as in Proposition \ref{Ker(mu):AOV1}. Thus 
$\alpha|_{\Aut_\call(P)}$ is conjugation by $g_P$. Set $g=g_{Q_2}\in 
C_{Z(Q_2)}(\Aut_S(Q_2))=Z(S)$. Upon replacing $\alpha$ by 
$c_g^{-1}\circ\alpha$, we can arrange that $g_{Q_2}=1$, and hence that 
$\alpha$ is the identity on $\Aut_\call(Q_2)$. Since $Z(S)=Z(N_G(S))$ 
(recall $Z(S)$ has type $\2{abc}$), $\alpha$ is still the identity on 
$\Aut_\call(S)$. 

Set $P_j=C_S(W_j)$ ($j=1,2$). By Proposition \ref{Ker(mu):AOV1}(c) and 
since $\sigma_j$ normalizes $P_j$ and $Q_2$, $g_{P_1}\equiv g_{Q_2}=1$ 
modulo $\gen{W_1\cap\2a}$, and $g_{P_2}\equiv g_{Q_2}=1$ modulo 
$\gen{W_2\cap\2c}$. Also, $\gen{W_1\cap\2a}\le Z(N_G(P_1))$ and 
$\gen{W_2\cap\2c}\le Z(N_G(P_2))$ (since $N_G(P_i)\le N_G(W_i)$). Thus 
$\alpha|_{\Aut_\call(P_j)}=\Id$ for $j=1,2$, so $\alpha=\Id$ by Proposition 
\ref{Ker(mu):AOV1}(d). This proves that $\Ker(\mu_G)=1$.

\boldd{When $G\cong\Fi_{24}'$,} since $W\ge Z(S)$, it contains at least one 
$\2b$-element (recall $\2a=\calj_2$ and $\2b=\calj_4$). By 
Propositions 3.3.1, 3.3.3, 3.4.1, and 3.4.2 in \cite{Wilson-Fi} (corrected 
in \cite[\S\,2]{LW-Fi24}), the normalizer of every elementary abelian 
2-subgroup of $G$ is contained up to conjugacy in $K_1$, $K_2$, or one of 
the $H_i$ for $i\le5$, except when it is $\2a$-pure and the symplectic form 
described above is nonsingular. So we can assume that $N_G(W)$ is contained 
in one of these groups. Together with earlier remarks, this means that we 
can eliminate all of the $K_i$, and assume that $N_G(W)\le H_i$ for some 
$1\le i\le5$. So $P\ge Q_i$ and $W\le V_i$, and $i\ne1$ since $\rk(V_1)=1$. 

By Lemma \ref{H/Q-essential}(b,c), 
$W=C_{V_i}(P^*)$, where $P^*=\Aut_P(V_i)$ is an essential 2-subgroup of 
$H_i^*=\Aut_{H_i}(V_i)$. If $i=2,3,5$, then $P^*\in\IND1{H_i^*}$ by Lemma 
\ref{H/Q-essential}(b,c), and hence $[S:P]=2$ since $H_i\ge S$. So 
$W=Z_2(S)$ in these cases by Lemmas \ref{Z2(S)}(a) and \ref{[S:C(V)]=2}. 

If $i=4$, then $H_i^*\cong\Omega_8^-(2)$, and the conditions 
$P^*\in\EEE{H_4^*}$ and $\rk(C_V(P^*))\ge2$ imply that 
$N_G(W)\cong2^8.(2^{3+6}.(S_4\times3))$ (the stabilizer of an isotropic 
line and plane in the projective space of $V_4$). Hence $\rk(W)=2$ and 
$|P|=2^{19}$. By \cite[Table 15]{Wilson-Fi}, there are only two classes of 
four-groups in $G$ with centralizer large enough, one of type $\2{aab}$ 
(impossible since $\Aut(W)$ is not a 2-group), and the other $Z_2(S)$ of 
type $\2b^2$. Thus $\5\calz=\{Z_2(S)\}$, and $\Ker(\mu_G)=1$ by Proposition 
\ref{Ker(mu):AOV1}(e).

\end{spor}

\begin{spor}{$F_5$} By \cite[\S\,3.1]{NW-HN}, each 2-local subgroup of $G$ 
is contained up to conjugacy in one of the subgroups
	\[ H_1\cong2^{1+8}_+.(A_5\times A_5):2, \qquad
	H_3\cong2^3.2^2.2^6.(3\times L_3(2)), \qquad
	H_6\cong2^6\cdot U_4(2), \]
	\[ K_1\cong2\cdot \HS:2, \qquad 
	K_2\cong(A_4\times A_8):2 < A_{12}. \]
As usual, set $Q_i=O_2(H_i)$ and $V_i=Z(Q_i)$ for $i=1,3,6$. Then $V_1$ and 
$V_3$ are $\2b$-pure, and $O_2(K_1)$ and $O_2(K_2)$ are $\2a$-pure. 
By \cite[\S\,3.1]{NW-HN}, for each elementary abelian 2-subgroup $V\le G$, 
there is a quadratic form $\qq\:V\to\F_2$ defined by sending $\2a$-elements 
to 1 and $\2b$-elements to 0. 

Fix $W\in\5\calz$, and set $P=C_S(W)$. Then $W\ge Z(S)$, 
so $W\cap\2b\ne\emptyset$. So either the quadratic form $\qq$ on $W$ is 
nondegenerate and $\rk(W)\ge3$, or there is a $\2b$-pure subgroup $W_0\le 
W$ such that $N_G(W)\le N_G(W_0)$. By \cite[\S\,3.1]{NW-HN}, in this last 
case, $N_G(W_0)\le H_i$ for $i=1$ or $3$. 

If $N_G(W)\le N_G(W_0)\le H_i$ for $i=1,3$, then $P\ge O_2(H_i)$, so $W\le 
V_i$. In particular, $i\ne1$. If $N_G(W)\le H_3$, then $P$ has index 
$2$ in $S$ since $\Aut_G(V_i)\cong L_3(2)$, so $W=Z_2(S)$ by Lemmas 
\ref{[S:C(V)]=2} and \ref{Z2(S)}(a). 

Now assume $\qq$ is nondegenerate as a quadratic form (and $\rk(W)\ge3$). 
Let $W^*<W$ be a $\2a$-pure subgroup of rank $2$, and identify $C_G(W^*)$ 
with $(2^2\times A_8)<A_{12}<G$. If $\rk(W)=3$, then we can identify $W$ 
with $\gen{(1\,2)(3\,4),(1\,3)(2\,4),(5\,6)(7\,8)}$, so 
$C_G(W)\cong2^2\times(2^2\times A_4):2$, $P=C_S(W)\cong2^2\times(2^4:2)$, 
$Z(P)\cong2^4$, which contradicts the assumption that $W=\Omega_1(Z(P))$. 
If $\rk(W)\ge4$, then it must be conjugate to one of the subgroups (1), 
(2), or (3) defined in \cite[p. 364]{NW-HN} (or contains (2) or (3) if 
$\rk(W)=5$). Then $C_G(W)\cong E_{2^6}$ or $E_{16}\times A_4$, so $P=W\cj 
V_6$, which is impossible since $\Aut_G(V_6)\cong U_4(2)$ does not contain 
a strongly embedded subgroup.

Thus $\5\calz=\{Z_2(S)\}$, and $\mu_G$ is injective by Proposition 
\ref{Ker(mu):AOV1}(e).

\end{spor}

\begin{spor}{$F_3$} By \cite[Theorem 2.2]{Wilson-Th}, there are two classes 
of maximal $2$-local subgroups of $G$, represented by 
$H_1\cong2^{1+8}_+.A_9$ and $H_5\cong2^5.\SL_5(2)$. Set 
$Q_i=O_2(H_i)$ and $V_i=Z(Q_i)\cong E_{2^i}$ ($i=1,5$).

Fix $W\in\5\calz$, set $P=C_S(W)$, and let $i=1,5$ be such that 
$N_G(W)\le H_i$. Then $P\ge O_2(H_i)$ and $W\le V_i$, so $i=5$. By Lemma 
\ref{H/Q-essential}(b), $P/V_5\in\EEE{H_5/V_5}$ (where $H_5/V_5\cong 
L_5(2)$) and $[S:P]=2$. Hence $W\le Z_2(S)$ by Lemma \ref{[S:C(V)]=2}. 
Since $|Z_2(S)|=4$ by Lemma \ref{Z2(S)}, this proves that 
$\5\calz=\{Z_2(S)\}$, and hence that $\Ker(\mu_G)=1$ by Proposition 
\ref{Ker(mu):AOV1}(e).

\end{spor}

\begin{spor}{$F_2$, $F_1$} If $G\cong F_1$, then by \cite[Theorem 1]{MS}, there 
are maximal 2-local subgroups of the form
	\[ H_1\cong2^{1+24}.\Co_1,\quad
	H_2\cong2^2.[2^{33}].(M_{24}\times S_3),\quad
	H_3\cong2^3.[2^{36}].(L_3(2)\times3\cdot S_6), \] 
	\[ H_5\cong2^5.[2^{30}].(S_3 \times L_5(2)), \quad
	H_{10}\cong2^{10+16}\cdot\Omega_{10}^+(2), \]
If $G\cong F_2$, then by \cite[Theorem 2]{MS}, there 
are maximal 2-local subgroups of the form
	\[ H_1\cong2^{1+22}.\Co_2,\quad
	H_2\cong2^2.[2^{30}].(M_{22}:2\times S_3),\quad
	H_3\cong2^3.[2^{32}].(L_3(2)\times S_5), \] 
	\[ H_5\cong2^5.[2^{25}].L_5(2), \quad
	H_{9}\cong2^{9+16}\cdot\Sp_8(2), \]
As usual, we set $Q_i=O_2(H_i)$, and $V_i=Z(Q_i)\cong E_{2^i}$. In both 
cases ($G\cong F_1$ or $F_2$), $H_1=C_G(x)\ge S$ for $x\in\2b$, and $H_i>S$ 
($V_i\nsg S$) for each $i$.


Fix $W\in\5\calz$, and set $P=C_S(W)$. Then $W\ge Z(S)$, and hence 
$W$ contains $\2b$-elements. By \cite[Lemma 2.2]{Meierfrankenfeld}, $W$ is 
``of 2-type'', in the sense that $C_G(O_2(C_G(W)))$ is a 2-group, since the 
subgroup generated by a $\2b$-element is of $2$-type. In particular, 
$C_G(P)$ is a $2$-group and hence $C_G(P)=Z(P)$.

A $\2b$-pure elementary abelian $2$-subgroup $V\le G$ is called 
\emph{singular} if $V\le O_2(C_G(x))$ for each $x\in V^\#$. 
If $G\cong F_1$, then by \cite[Proposition 9.1]{MS}, applied with $P$ in the 
role of $Q$ and $t=1$, there is a subgroup $W_0\le W$ such that $N_G(W)\le 
N_G(W_0)$, and either $W_0$ is $\2b$-pure and singular or $W=W_0\cj V_{10}$. 
Since $\Aut_G(V_{10})\cong\Omega_{10}^+(2)$ has no strongly embedded 
subgroup, $W_0$ must be $\2b$-pure and singular, and hence 
$N_G(W)\le H_i$ for some $i=1,2,3,5$ by \cite[Theorem 1]{MS}. Thus $P\ge 
Q_i$ and $W\le V_i$, so $W$ is also $\2b$-pure and singular.

If $G\cong F_2$, identify $G=C_M(x)/x$, where $M\cong F_1$ and $x$ is a 
$\2a$-element in $M$. Let $\til{P}\le C_M(x)$ be such that $x\in\til{P}$ 
and $\til{P}/\gen{x}=P$, and set $\til{W}=\Omega_1(Z(\til{P}))$.  Then 
$x\in\til{W}$ and $\til{W}/\gen{x}\le W$, and 
$(W\cap\2b)\subseteq\til{W}/\gen{x}$ since $\2b$-elements in $G$ lift to 
pairs of involutions of classes $\2a$ and $\2b$ in $M$ (coming from a 
subgroup of type $\2{baa}$ in $Q_1<M$). By \cite[Proposition 9.1]{MS} 
again, applied with $\til{P}$ in the role of $Q$ and $t=x$, there is a 
subgroup $\til{W}_0\le\til{W}$ such that $N_M(\til{W})\le N_M(\til{W}_0)$, 
and either $\til{W}_0$ is $\2b$-pure and singular or 
$\til{W}=\til{W}_0\cj[M]V_{10}^{(M)}$. In the latter case, $W\cj V_9$, 
which is impossible since $\Aut_G(V_9)\cong\Sp_8(2)$ has no strongly 
embedded subgroup. Again, we conclude that $N_G(W)\le H_i$ for some 
$i=1,2,3,5$ \cite[Theorem 2]{MS}, and that $W\le V_i$ by Lemma 
\ref{H/Q-essential}(a) and hence is $\2b$-pure and singular.

By \cite[Lemma 4.2.2]{MS}, applied with $W=1$ (if $G\cong F_1$) or 
$W=\gen{x}$ ($G\cong F_2$), the automizer of a singular subgroup is its 
full automorphism group. Since $\GL_n(2)$ has no strongly embedded subgroup 
for $n\ne2$, this implies that $\rk(W)=2$. By \cite[Lemma 4.4]{MS}, if we 
identify $Q_1/V_1\cong E_{2^{24}}$ with the mod $2$ Leech lattice, then 
$\2a$-elements correspond to the $2$-vectors and $\2b$-elements to the 
classes of $4$-vectors, and hence $H_1/Q_1\cong\Co_1$ acts transitively on 
each. So $F_1$ contains a unique class of singular subgroups of rank $2$. A 
similar argument, using \cite[Corollary 4.6]{MS}, now shows that $F_2$ also 
contains a unique class of singular subgroup of rank $2$. Since $H_2>S$, 
each of these classes has a representative normal in $S$, so $W\nsg S$, and 
$W=Z_2(S)$ by Lemmas \ref{Z2(S)}(a) and \ref{[S:C(V)]=2}.

To conclude, we have now shown that $\5\calz=\{Z_2(S)\}$ in both cases. 
Hence $\Ker(\mu_G)=1$ by Proposition \ref{Ker(mu):AOV1}(e).

\end{spor}

This finishes the proof of Proposition \ref{mu-2-big}.
\end{proof}

By inspection in the above proof, in all cases where $Z_2(S)\cong E_4$ and 
its involutions are $G$-conjugate, we have $\5\calz(\calf)=\{Z_2(S)\}$. A 
general result of this type could greatly shorten the proof of Proposition 
\ref{mu-2-big}, but we have been unable to find one. The following example 
shows that this is not true without at least some additional conditions.


Set $G=2^4:15:4\cong\F_{16}\rtimes\GGL_1(16)$. Set $E=O_2(G)\cong E_{16}$, 
fix $S\in\syl2{G}$, and let $P\nsg S$ be the subgroup of index $2$ 
containing $E$. Then $Z_2(S)=Z(P)\cong E_4$ and $\Aut_G(Z_2(S))\cong S_3$, 
while $\5\calz(\calf_S(G))=\{Z_2(S),E\}$.

\end{document}